\documentclass[11pt,a4paper]{article}

\usepackage{cite}                                       
\usepackage[latin1]{inputenc}
\usepackage[english]{babel}
\usepackage{amssymb,amsmath}
\usepackage{a4wide}
\usepackage{enumitem}

\newcommand{\A}{\mathcal A}
\newcommand{\B}{\mathcal B}

\newcommand{\Z}{\mathcal Z}
\newcommand{\Ph}{\mathcal P}

\newcommand{\R}{\mathcal R}

\newcommand{\T}{\mathcal T}

\newcommand{\Nscr}{\mathcal{N}}
\newcommand{\Pscr}{\mathcal{P}}
\newcommand{\Rscr}{\mathcal{R}}
\newcommand{\Xscr}{\mathcal{X}}
\newcommand{\Dscr}{\mathcal{D}}

\newcommand{\tond}[1]{{\left(#1\right)}}
\newcommand{\quadr}[1]{{\left[#1\right]}}

\newcommand{\qed}{{\vskip-18pt\null\hfill$\square$\vskip6pt}}
\newcommand{\Ham}[2]{H^{(#1)}_{#2}}
\newcommand{\Poi}[2]{\{#1,#2\}}
\newcommand{\derp}[2]{{\frac{\partial #1}{\partial #2}}}

\newcommand{\ncamp}[1]{\Big\|#1\Big\|^\oplus}
\newcommand{\norm}[1]{\left\|#1\right\|}

\def\range{\Rscr^{(s)}}
\def\pspazio{\Pscr^{(s)}}

\def\CC{\mathbb{C}}
\def\RR{\mathbb{R}}
\def\ZZ{\mathbb{Z}}

\def\Id{\mathbb I}

\def\Im{i}
\def\lie#1{L_{#1}}
\def\chiph{\chi^{\phantom s}}
\def\gt{>}
\def\lt{<}
\def\le{\leq}
\def\ge{\geq}

\def\Chi{\Xscr}

\newtheorem{theorem}{Theorem}[section]
\newtheorem{theorem*}{Theorem}
\newtheorem{lemma}{Lemma}[section]
\newtheorem{corollary}{Corollary}[section]
\newtheorem{proposition}{Proposition}[section]
\newtheorem{definition}{Definition}[section]
\newtheorem{remark}{Remark}[section]

\newenvironment{proof}[0]{\noindent{\bf proof:}}{$\square$\par\medskip}

\title{An extensive resonant normal form\\
  for an arbitrary large Klein-Gordon model}

\author{Simone Paleari and Tiziano Penati}

\begin{document}



\maketitle

\begin{abstract}
  We consider a finite but arbitrarily large Klein-Gordon chain, with
  periodic boundary conditions. In the limit of small couplings in the
  nearest neighbor interaction, and small (total or specific) energy,
  a high order resonant normal form is constructed with estimates
  uniform in the number of degrees of freedom. In particular, the
  first order normal form is a generalized discrete nonlinear
  Schr\"odinger model, characterized by all-to-all sites coupling with
  exponentially decaying strength.

\end{abstract}


\section{Introduction and statement of the results}
\label{s:1}

In the present paper, along the lines of~\cite{Car07, CarM12, GioPP12,
  GioPP13, MaiBC13}, we keep on investigating the development and
application of a perturbation theory for Hamiltonian systems with an
arbitrarily large number of degrees of freedom, and in particular in
the thermodynamic limit. Indeed, motivated also by the problems
arising in the foundations of Statistical Mechanics, we want to
consider large systems (e.g. for a model of a crystal the number of
particles should be of the order of the Avogadro number) with non
vanishing energy per particle (which corresponds to a non zero
temperature in the physical model).

Since we are interested in the low temperature regime (aiming for
example at some rigorous results of the classical mechanics
description of the behavior of the specific heats in such a regime),
it is foreseeable the use of perturbation theory to exploit the
presence of a small parameter like the specific energy. Unfortunately,
it is a well known limit of the classical results of this theory (like
KAM or Nekhoroshev theorem) to suffer a bad dependence on the number
of degrees of freedom, often resulting in void or non applicable
statements in the thermodynamic limit.

In the recent papers~\cite{CarM12,GioPP13,MaiBC13} it has been
possible to prove, for the first time, the existence of an approximate
conserved quantity, independent of the Hamiltonian, exactly in the
thermodynamic limit, thus with uniform estimates in the number of
degrees of freedom and non vanishing specific energy.

In the present work we make progress in the above mentioned research
program, with a result in the direction of a normal form construction,
rather than in that of approximate conserved quantities. Such a
construction is shown to hold both in a regime of small total energy
and in a regime of small specific energy; moreover we are able to
completely control the dependence on the number of degrees of freedom,
thanks to the ideas and techniques used in~\cite{GioPP13}, that we
here extend also to cover the normal form algorithm.

We consider a Klein-Gordon model as described by the following Hamiltonian
\begin{equation}
  \label{e.H}
  H(x,y) = \frac12\sum_{j=1}^N \left[ y^2_j + x^2_j + a(x_{j+1}-x_j)^2 \right]
  + \frac14\sum_{j=1}^Nx_j^4 \ ,
\qquad
  x_0=x_N , \ y_0=y_N \ ,
\end{equation}
i.e. a finite chain of $N$ degrees of freedom and periodic boundary
conditions.

Our main result holds in the limit of small coupling (the constant $a$
in the Hamiltonian), and small (both total and specific) energy. A
simplified statement of our normal form construction could be the
following (see Proposition~\ref{p.1} and Theorem~\ref{prop.gen} for
the complete ones)

\vskip10pt
\noindent{\bf Theorem\ } {\itshape There exist $C_1$ and $C_2$, such
  that for every $N$, every small enough value of the coupling
  constant $a$, every integer $r<C_1/a$, and (total or specific)
  energy less than $C_2/r^2$, there exists an analytic canonical
  transformation, under which the Hamiltonian~\eqref{e.H} takes the
  form
  \begin{equation*}
    \Ham{r}{} = H_\Omega + Z_0 + \dots + Z_r + {P^{(r+1)}}\ ,
    \qquad\qquad
    \{H_\Omega,Z_s\}=0 \quad\forall s\in\{0,\ldots,r\}\ .
  \end{equation*}
  with $H_\Omega$ a system of $N$ identical harmonic oscillators,
  $Z_s$ homogeneous polynomials of order $2s+2$, $P^{(r+1)}$ a
  remainder of order $2r+4$ and higher.  } \vskip10pt

The first aspect we need to remark here is that, in order to control
as in~\cite{GioPP13} the dependence on $N$ in the whole perturbation
construction, we exploit the fact that the Hamiltonian is extensive,
i.e. the energy of the system is, roughly speaking, proportional to
$N$. The extensivity is the result of two general properties of our
(and similar) model, which are introduced and discussed in general in
Section~\ref{s:extensivity}. The first one is the translational
invariance, that we formalize through what we call cyclic symmetry
(see definition~\ref{d.cs} and in general subsection~\ref{ss:form}):
as it happens in a cyclic group, we exploit this discrete symmetry by
introducing the idea of a generator $f$ (a ``seed'' in our
terminology, see \eqref{e.cycl-fun}) of a translation-invariant
function $F$. The second property is the short interaction range (see
subsection~\ref{ss.int.range}): actually~\eqref{e.H} in the original
variables possesses a finite range (nearest neighbor) interaction,
which is immediately replaced, in the perturbation construction, by an
infinite range interaction with an exponentially decaying strength
with the distance. These properties of cyclic symmetry and exponential
decay in the interaction range are explicitly controlled in particular
for the whole transformed Hamiltonian in Theorem~\ref{prop.gen},
allowing us to obtain estimates uniform with $N$. Indeed, it is
possible to implement the whole perturbation scheme at the level of
the seeds, whose norms (due to the short range interaction) are
independent\footnote{Although it is not our aim to give here a formal
  definition of ``extensivity'', we could associate the concept of
  being ``extensive'' for a function $F$, with the following two
  properties: for $F$ to be cyclically symmetric and for its
  generator/seed $f$ to have a norm independent of $N$. Indeed this
  could be a rather general characterization, which is meaningful for
  every function. And if in particular $F$ represents the potential of
  an interaction force, then, since the independence of $\norm{f}$
  with respect to $N$ is equivalent to the sufficiently fast decay of
  the interaction with the distance among sites, one could recover the
  usual idea of extensivity given by translation invariance plus short
  range interaction.} of $N$. In our opinion this represents the main
original aspect in the use of the translational discrete symmetry
inside a perturbation construction for an Hamiltonian chain. The
possibility to preserve a discrete symmetry while performing a normal
form construction is surely not completely new: as an example, we
could mention \cite{Rink01,Rink03}, where the dyhedral group
symmetries have been successfully exploited to show the Liouville
integrability of the Birkhoff-Gustavson normal form for a periodic FPU
lattice. However, the benefits of our strategy of working at the level
of the ``seeds'' in the perturbation scheme go beyond the information
on the structure of the normal form: it is the key ingredient to even
provide estimates, also with a sharp dependence on the number of
degree of freedom of the system.

Another aspect to remark is the validity of the normal form Theorem
both in specific and total energy regime: and if in particular we
restrict to the small total energy regime some dynamical information
can be immediately obtained, since it is possible to deduce a somewhat
complementary result to that of~\cite{GioPP13}. Indeed, in
Corollary~\ref{c.Hom.K.var} (see Section~\ref{s:results}), we have
provided a long time adiabatic invariance of the $\ell^2$ norm
$H_\Omega$ in the classical sense (i.e. not in the probabilistic
formulation of~\cite{GioPP13}). An analogue result in the form of a
dynamical application of our Theorem in the specific energy regime is
instead a much more difficult task, and reasonable results in that
direction are still missing.

It is nevertheless worth to put more into evidence other potential
advantages of the present normal form construction, which gives more
information on the structure of the Hamiltonian. Indeed, in
Section~\ref{s:application} we concentrate on the low order normal
form $H_\Omega + Z_0 + Z_1$ (hence choosing $r=1$ in the Theorem),
which represents a generalized discrete nonlinear Schr\"odinger
(GdNLS) chain, characterized by all-to-all sites couplings, both in
the linear and nonlinear terms, with exponential decay of the
coefficients with the distance between sites.  Since it turns out that
such a decay is given by powers of the small coupling constant $a$, a
truncation of this normal form results in the usual dNLS model, which
is well known to provide a leading order approximation of the KG
evolution, both in the continuum limit (as for example
in~\cite{BamCP09}) and in the anticontinuum limit (as discussed for
example in \cite{PelS12} or formally used in the modulational
instability description like in \cite{DauDP97,ClaKKS93}).

We think that such a GdNLS model can be interesting by itself: indeed,
due to the approach used in its construction, it contains the right
normal form for several different regimes with respect to the relative
smallness between the two small parameters, the coupling and the
energy.

We also observe that, in the regime of small energy, the result is
valid also in the case of the soft nonlinearity, i.e. with a minus in
front of the quartic term; we stress that, in such a case, the second
order normal form $H_\Omega + Z_0 + Z_1 + Z_2$ could be seen as a
perturbation (again due to the all-to-all sites couplings) of the
cubic-quintic dNLS (see, e.g.,~\cite{ChoP11,CarTCM06}), with competing
nonlinearities.

We close the comments on the GdNLS with the remark that it could be
the starting point for several applications which concern the
Klein-Gordon model, like the variational approximations for breather
solutions (see, e.g.,~\cite{ChoPS12,ChoP11}), the approximation of the
small amplitude Cauchy problem, the existence and linear stability of
multibreathers (see, e.g.,~\cite{KouI02,PelKF05,KouK09,PelS12}).
Moreover there are several recent works on models with more than first
neighbor interactions or with different nonlinearities,
like~\cite{KouKCR13,ChoCMK11,Rap13,Yos12} where spatially localized
periodic orbits, as breathers or multibreathers, are studied: with
respect to this, we expect that the approach proposed in the present
paper can be suitably extended in these more general models, leading
to different GdNLS-like normal forms. As an example, in a subsequent
paper~\cite{PP14b}, following also some ideas from~\cite{BamN98} and
still in the spirit of dealing with the case $N<\infty$ (see
\cite{PP12} for a dNLS study), we apply this normal form construction
to a mixed KG/FPU model in order to provide a result of long time
approximation of the dynamics close to site-symmetric breather
solutions of a GdNLS model, whose existence and stability are proved
as intermediate step.

As a last remark on the potential benefit of the present normal form
and of the techniques developed and used, we recall the realm of
numerical and/or symbolic computations. Indeed it would be interesting
to figure out how to exploit our formalization of the extensivity in
infinite systems (e.g. some classes of PDEs) in order to perform
algebraic manipulations for these type of systems.

The paper is structured as follows. In Section~\ref{s:extensivity} we
recall the formalization of the the physical properties of the model,
including some results of~\cite{GioPP13} and adding the control of
Hamiltonian vector fields.  In Section~\ref{s:results} we present and
further comment the results of the paper in a more detailed and
complete form. In Section~\ref{s:application} we discuss the GdNLS
normal form. In Section~\ref{s:construction} the formal construction
is presented and most of the proofs of the estimates are given, with
further technical proof deferred to the Appendix~\ref{s:6}.

%
%
%
%

\section{A formalization of extensivity: cyclic symmetry and short
  range interaction}
\label{s:extensivity}

This Section is devoted to the formalization of two fundamental,
though quite general, properties of the KG chain, i.e. the discrete
translation invariance that we call \emph{cyclic symmetry}, and the
short interaction range. Most of the definitions and the results have
been already introduced in~\cite{GioPP12,GioPP13}: we repeat them here
in Subsections~\ref{ss:form} and \ref{ss.int.range}, with a hopefully
more terse exposition, to let the present paper be self-contained. In
Subsection~\ref{ss.Hvf} we add some new results , i.e. the treatment
of vector fields, which were not necessary in our previous papers.

The system under investigation, i.e. the finite but large KG chain
described by the Hamiltonian~\eqref{e.H}, possesses some general
properties shared by a lot of many particle systems, as discussed in
\cite{GioPP13}: they are characterized by two-body conservative forces
with smooth potential which are invariant with respect to rotations
and/or translations.  These properties are quite general ones.

Since we will apply our construction on the KG chain, we restrict our
attention to a system of identical particles on a $d$--dimensional
lattice, with a short or even finite range interaction\footnote{{We
    recall that by ``finite range interaction'' we mean an interaction
    which, for any particle, involves only a finite number of
    neighbors, independent of $N$; a ``short range interaction''
    instead may involve even an all-to-all interaction provided its
    strength decays fast enough with the distance.}}.  Moreover it is
enough to know the local interaction of a particle with its neighbors
and the complete Hamiltonian is the sum of the contribution of every
particle to both the kinetic and the potential energy.  We thus have
the presence of both such a cyclic symmetry, and of a short range
interaction potential. These properties characterize and somewhat
formalize the fact that the Hamiltonian is {\sl extensive}, i.e. it is
proportional to the number of degrees of freedom $N$.  Functions
possessing the same extensivity property of the Hamiltonian are
particularly relevant.

\subsection{Cyclic symmetry}
\label{ss:form}

We consider the simplified model of a finite one dimensional lattice
with periodic boundary conditions and finite range interactions
(nearest neighbors). We denote by $x_j,\,y_j$ the position and the
momentum of a particle, with $x_{j+N} = x_j$ and $y_{j+N}=y_j$ for any
$j$.

\paragraph{Cyclic symmetry. }
\label{p:cyclic}

We formalize one of the ingredient of the extensivity, i.e. discrete
translation invariance, by using the idea of \emph{cyclic
  symmetry}. In~\cite{GioPP12,GioPP13} we decided to use such a new
and nonstandard terminology both to remind that the associated group
is a cyclic one, and also because we think we exploited the invariance
in a novel way. The \emph{cyclic permutation} operator $\tau$, acting
separately on the variables $x$ and $y$, is defined as
\begin{equation}
\label{e.perm}
\tau(x_1,\ldots,x_N) = (x_2,\ldots,x_N,x_1)\ ,\quad
\tau(y_1,\ldots,y_N) = (y_2,\ldots,y_N,y_1)\ .
\end{equation}
We extend its action on the space of functions as
\begin{equation*}
\bigl(\tau f\bigr)(x,y) = f(\tau(x,y)) = f(\tau x,\tau y) \ .
\end{equation*}

\begin{definition}
\label{d.cs}
We say that a function $F$ is \emph{cyclically symmetric} if
$\tau F = F$.
\end{definition}

In order to further exploit the symmetry, we now try to formalize the
idea that, from the point of view of information content, there is a
lot of redundancy in a cyclically symmetric function, i.e. it can be
reconstructed from a ``smaller'' object.  We thus introduce an
operator, indicated by an upper index $\oplus$, acting on functions:
given a function $f$, a new function $F= f^{\oplus}$ is constructed as
\begin{equation}
  F= f^{\oplus} := \sum_{l=1}^{N} \tau^l f \ .
\label{e.cycl-fun}
\end{equation}
We shall say that $f^{\oplus}(x,y)$ is generated by the \emph{seed}
$f(x,y)$. We will try to use the convention of denoting cyclically
symmetric functions with capital letters and their seeds with the
corresponding lower case letter.

\begin{lemma}[see \cite{GioPP13}]
The following holds:
\begin{enumerate}
\item given a seed $f$, then for $F= f^{\oplus}$ one has $\tau F = F$;
\item given a function $F$, such that $\tau F = F$, then there exist a
  (not unique) seed $f$ such that $F= f^{\oplus}$;
\item for any integers $s_1,\,s_2$, $(f_1+f_2)^{\oplus} =
  \tond{\tau^{s_1} f_1 +\tau^{s_2}f_2}^{\oplus}$;
\item the Poisson bracket between two cyclically symmetric functions is
      also cyclically symmetric, i.e. we may write $h^{\oplus} =
      \Poi{f^\oplus}{g^\oplus}$.  A candidate seed is $h =
      \Poi{f}{g^\oplus}$.
\end{enumerate}
\end{lemma}
It is worth to stress that property 4 of the above Lemma is the one
which allows to perform the normal form construction by preserving the
cyclic symmetry. From a purely formal point of view, the compatibility
of the discrete translation invariance with a canonical perturbation
construction is not new, since the possibility to perform the
Lie-transform normalization by preserving a discrete (and symplectic)
symmetry is a well known fact (see for example\cite{Rink01,Rink03} and
references therein). Once again we bring this fact down the level of
seeds to exploit it further.

\paragraph{Polynomial norms. }
\label{p:polinorms}

Let $f(x,y)=\sum_{j,k} f_{|j|+|k|=s} x^j y^k$ be a homogeneous
polynomial of degree $s$ in $x,\,y$.  Given a positive $R$, we define
its polynomial norm as
\begin{equation}
\label{e.polinorm}
\|f\|_R := R^s \sum_{|j|+|k|=s} |f_{j,k}|\ .
\end{equation}
If $R$ represents the radius of the ball centered in the origin of the
phase space, endowed for example with the euclidean norm, then one
would have
\begin{displaymath}
|f(x,y)|\leq \sup_{\norm{(x,y)}\leq R}|f(x,y)| \leq \|f\|_R\ .
\end{displaymath}

\paragraph{Norm of a cyclically symmetric function. }
\label{p:norm-ext}

Assume now that we are equipped with a norm for our functions
$\norm{\cdot}$, e.g. the above defined polynomial norm. We introduce a
corresponding norm $\|\cdot\|^{\oplus}$ for a cyclically symmetric function
$F=f^\oplus$ by defining
\begin{equation}
\label{e.norm-germ}
\bigl\|F\bigr\|^{\oplus} = \|f\|\ ,
\end{equation}
i.e. we actually measure the norm of the seed. An obvious remark is
that the norm so defined depends on the choice of the seed, but this
will be harmless in the following.  The relevant facts are the
following:
\begin{lemma}[see \cite{GioPP13}]It holds:
\begin{enumerate}
\item for any $s$ one has $\norm{\tau^s f} = \norm{f}$; 
\item  the inequality
$\|f^{\oplus}\| \le N \bigl\|f^{\oplus}\bigr\|^{\oplus}$
holds true for any choice of the seed. 
\end{enumerate}
\end{lemma}

\noindent
This is particularly useful when the norm of the seed turns out to be
independent of $N$, since in such a case the dependence on the number
of degrees of freedom is completely factorized and totally under
control. Moreover one could verify if a function is ``extensive'' by
checking if it is cyclically symmetric and with its cyclic norm
independent of $N$. We remark that, if the function under
consideration is an interaction potential, then obviously this second
property is equivalent to the interaction range\footnote{more on this
  point in Subsection~\ref{ss.int.range}} being short: indeed, it is
possible to get a seed whose norm \eqref{e.polinorm} is independent on
$N$ also in the case of an all-to-all interaction like
\begin{displaymath}
\sum_{i,j}a_{i,j}(x_i-x_j)^2\ ,
\end{displaymath}
provided some suitable decay of the coefficients $|a_{i,j}|$ with the
distance $|i-j|$.

\paragraph{Circulant matrices. }
\label{p:circulant}

When we deal with particular functions which are quadratic forms, the
cyclic symmetry assumes a particular form.  Let us thus restrict our
attention to the harmonic part of the Hamiltonian: it is a quadratic
form represented by a matrix $A$
\begin{equation}
\label{e.H0}
H_0(x,y) = \frac12 y\cdot y + \frac12 Ax\cdot x.
\end{equation}
If the Hamiltonian $H_0$ is cyclically symmetric, then
$H_0=h_0^\oplus$.  This implies that $A$ commutes with the matrix
$\tau$ representing the cyclic permutation~\eqref{e.perm}
\begin{equation}
\label{e.tau}
\tau_{ij}=
\begin{cases}
1\quad {\rm if}\ i=j+1\>({\rm mod}\,N)\>,\\
0\quad \rm{otherwise}.
\end{cases}
\end{equation}
We remark that the matrix $\tau$ is orthogonal and generates a cyclic
group of order $N$ with respect to the matrix product.

We recall the following
\begin{definition}
\label{d.circulant}
A matrix $A\in {\rm Mat}_{\RR}(N,N)$ is said to be \emph{circulant} if
\begin{equation}
\label{e.circ.1}
A_{j,k} = a_{(k-j)\> ({\rm mod}\,N)}\ .
\end{equation}
\end{definition}
Actually, the set of circulant matrices is a subset of Toepliz
matrices, i.e those which are constant on each diagonal. For a
comprehensive treatment of circulant matrices, see, e.g.,
\cite{Dav79}.
We just remind some properties that will be useful later.

\begin{enumerate}
\item The set of $N\times N$ circulant matrices is a real vector space
  of dimension $N$, and a basis is given by the cyclic group generated
  by $\tau$.
\item The set of matrices which commute with $\tau$  coincides with
  the set of circulant matrices.
\item The set of eigenvalues of a circulant matrix is the Discrete
  Fourier Transform of the first row of the matrix and viceversa.
\item Let $M^2=A$, where $A$ is circulant; then $M$ is circulant,
  too. Moreover, from the definition of $M:=\sqrt{A}$, it follows that
  if $A$ is symmetric, then $M$ is also symmetric.
\end{enumerate}

In our problem the cyclic symmetry of the Hamiltonian implies that the
matrix $A$ of the quadratic form is circulant. Obviously it is also
symmetric, so that the space of matrices of interest to us has
dimension~$\left\lfloor\frac{N}2\right\rfloor+1$. Indeed, a circulant
and symmetric matrix is completely determined by
$\left\lfloor\frac{N}2\right\rfloor+1$ elements of its first line.

\def\supp{\mathop{\rm supp}}
\def\diam{\mathop{\rm diam}}
\def\corsivo#1{{\sl #1}}

\subsection{Interaction range}
\label{ss.int.range}

Besides the translation invariance, usually the second ingredient for
the formalization of extensivity is the sufficiently fast decay of the
interaction strenght\footnote{At least for those function for which
  the concept makes sense, i.e. those giving an interaction
  potential.}, which is equivalent to the independence on $N$ of the
cyclic norm of the interaction potential. We give here some
definitions and properties at the level of the functions' seeds.  We
restrict our analysis to the set of polynomial functions.  We start
with some definitions. Let us label the variables as $x_l,y_l$ with
$l\in\ZZ$, and consider a monomial $x^jy^k$ (in multiindex notation).
\begin{definition}
We define the \emph{support} $S(x^jy^k)$ of the
monomial and the \emph{interaction distance} $\ell(x^jy^k)$ 
as follows: considering the exponents $(j,k)$ we set
\begin{equation}
\label{e.supp}
S(x^jy^k) = \{l\>:\>j_l\neq0 {\rm\ or\ } k_l\neq0 \}\ ,\quad
\ell(x^jy^k) = \diam\bigl(S(x^jy^k)\bigr) \ .
\end{equation}
We say that the monomial is \emph{left aligned} in case
$S(x^jy^k)\subset \{0,\ldots,\ell(x^jy^k)-1\}$.
\end{definition}
The definitions above is extended to a homogeneous polynomial $f$ by
saying that $S(f)$ is the union of the supports of all the monomials
in $f$, and that $f$ is left aligned if all its monomials are left
aligned.  The relevant property is that if $\tilde f$ is a seed of a
cyclically symmetric function $F$, then there exists also a left
aligned seed $f$ of the same function $F$: just left align all
monomials in $\tilde f$.

\paragraph{Short range (exponential decay of) interaction.}
\label{p:shortrange}
Although it is not necessary for the interaction to be short, we
consider the case of an exponential decay of the interaction strength,
since in our case this is the property which holds.  For the seed $f$
of a function consider the decomposition
\begin{equation}
\label{e.decomp}
f(z) =  \sum_{m\ge 0} f^{(m)}(z)\ ,\quad
 f^{(m)}(z) = \sum_{\ell(k)\le m} f_k z^k\ ,
\end{equation}
assuming that every $f^{(m)}$ is left aligned.
\begin{definition}
The seed $f$ (of a cyclically symmetric function) is of class
$\Dscr(C_f,\sigma)$ if
\begin{equation}
\label{dcdm.5}
\norm{f^{(m)}}_1 \le C_f e^{-\sigma m}\ ,\quad C_f\gt 0\,,\> \sigma\gt 0\ .
\end{equation}
\end{definition}

\begin{remark}
It is immediate to notice that when $C_f$ does not depend on $N$, then
\begin{displaymath}
\norm{f}_1 \leq \sum_{m\geq 0}\norm{f^{(m)}}_1 \leq C_f \sum_{m\geq
  0}e^{-\sigma m} = \frac{C_f}{1-e^{-\sigma }}\ ,
\end{displaymath}
hence $\norm{f}_1$ does not grow with $N$.
\end{remark}

%
%
%
%

\subsection{Hamiltonian vector fields}
\label{ss.Hvf}

We introduce here some definitions and some results concerning
Hamiltonian vector fields, their Lie derivatives, and the control of
their norms. This part is completely absent in \cite{GioPP13} since in
such a paper all the perturbation construction is performed at the
level of the Hamiltonian functions and not at the level of the vector
fields.

We consider, as an Hamiltonian, a cyclically symmetric function $F$
with seed $f$; we will make use of the common notation\footnote{For an
  easier notation we drop the Hamiltonian $F$ in the indexes of the
  components of the vector field.} $X_F=(X_1, \ldots,X_N,X_{N+1},
\ldots, X_{2N})$ to indicate the associated Hamiltonian vector field
$J\nabla F$, with $J$ given by the Poisson structure. The first easy,
but important, result is that also the Hamiltonian vector field
inherits, in a particular form, the cyclic symmetry; a possible choice
for the equivalent of the seed turn out to be the pair $(X_{1},
X_{N+1})$, i.e. the first and the $(N+1)^{\rm th}$ components of the
vector. This fact, which will be more clear thanks to the forthcoming
Lemma~\ref{l.seme.campo}, allows us to define in a reasonable and
consistent way the following norm
\begin{equation}
\label{e.def1}
\ncamp{X_F}_R := \norm{X_1}_R+\norm{X_{N+1}}_R \ .
\end{equation}

\begin{lemma}
\label{l.seme.campo}
Given $F=f^\oplus$, for the components of its Hamiltonian vector field
$X_F$ we have\footnote{An immediate consequence of
  \eqref{e.seme.campo} is that, defining the norm of the vector field
  as the sum of its components (i.e. a finite $\ell^1$ norm), we would
  get $\norm{X_F}_R = N\ncamp{X_F}_R$, which in turn justify the
  definition \eqref{e.def1}, and make it consistent with our previous
  definition~\eqref{e.norm-germ}.}
\begin{equation}
\label{e.seme.campo}
\begin{aligned}
 X_j &= \tau^{j-1} X_1
\cr
 X_{N+j} &= \tau^{j-1} X_{N+1}
\end{aligned}
\qquad\qquad
j=1,\ldots,N \ .
\end{equation}
Moreover, it holds
\begin{equation}
\label{e.xxx}
 \ncamp{X_F}_R = \sum_{l=1}^{2N}\norm{\derp{f}{z_l}}_R.
\end{equation}
\end{lemma}

\begin{proof}
We start by observing the following identity about the commutation
properties of partial derivative and cyclic permutation defined in
\eqref{e.perm}:
\begin{equation*}
 \frac{\partial}{\partial x_j} \quadr{f\circ\tau^l}
  = \tau^{-l}\quadr{\frac{\partial f}{\partial x_{j+l}}} \ ,
\end{equation*}
where, as usual, all the index for the variables are meant modulo $N$,
independently for each set $x$ and $y$. The similar relation holds for
the partial derivatives with respect to the $y$ variables.  

Using that $F= \sum_l \tau^lf$, and using the above relation to
``extract'' a permutation $\tau^{1-j}$, we have for $j=1,\ldots,N$
\begin{equation*}
\begin{aligned}
 X_j \equiv \frac{\partial F}{\partial y_j}
  =  \sum_l \frac{\partial}{\partial y_j} \quadr{\tau^lf}
 &= \sum_l \tau^{j-1}\quadr{\frac{\partial}{\partial y_{1}}
                            \tond{\tau^{l+j-1}f}}
  =
\\
 &= \tau^{j-1} \sum_m \frac{\partial}{\partial y_1} \quadr{\tau^mf}
  =  \tau^{j-1} \frac{\partial F}{\partial y_1}
\end{aligned}
\end{equation*}
which gives the first of \eqref{e.seme.campo}. Analogously for
$X_{N+j}$.  Concerning the equality~\eqref{e.xxx} one uses again the
commutation properties stated at the beginning of the proof, and then
the invariance of the polynomial norm $\norm{\cdot}_R$ under the
action of $\tau$.
\end{proof}

\begin{definition}
We denote with $\Ph$ the phase $\left(\RR^{2N},\norm{\cdot}\right)$,
endowed by either the euclidean norm ($\ell^2$) or the supremum norm
($\ell^\infty$). When necessary, we will specify the norm used with a
subscript, i.e. $\Ph_2$ with $\norm{\cdot}_2$ and $\Ph_\infty$ with
$\norm{\cdot}_\infty$.
\end{definition}

It is easy to check that, when dealing with ``local'' potentials like
$V(x)=\sum_j \frac1{2r}x_j^{2r}$, the corresponding Hamiltonian field
$X_V$ fulfills
\begin{displaymath}
\norm{X_V(x,y)} \leq  \norm{x}^{2r-1}\ ;
\end{displaymath}
with both the above introduced norms. Our aim is to generalize the
above estimate to cyclically symmetric Hamiltonian fields $X_F$ with
$\ncamp{X_F}_1<\infty$. To motivate the forthcoming
definition~\eqref{e.def.op.norm}, we remark that for any polynomial
vector field $X(z)$ of degree $r$ there exists a $r$-linear operator
$\tilde X(z_1,\ldots,z_r)$ such that
\begin{equation}
\label{e.multilin.X}
X(z) = \tilde X(z,\ldots,z)\ .
\end{equation}

\begin{definition}
For a polynomial vector field $X$ of degree $r$ define an ``operator
norm''
\begin{equation}
\label{e.def.op.norm}
\norm{X}_{\rm{op}}:=\sup_{\norm{z}\not=0}\frac{\norm{X(z)}}{\norm{z}^r}\ ,
\end{equation}
where, on the right hand side, all the $\norm{\cdot}$ can be either
$\norm{\cdot}_2$ or $\norm{\cdot}_\infty$.
\end{definition}

The following result, whose proof is deferred to the
Appendix~\ref{aa:pfield}, gives the above claimed control of the
cyclically symmetric Hamiltonian fields. We stress that it is valid
both in $\Ph_2$ and in $\Ph_\infty$.

\begin{proposition}
\label{p.field}
Let $f$ be an homogeneous polynomial of degree $r+1$ with $r\geq 1$
and $F=f^\oplus$ the cyclically symmetric Hamiltonian generated by
$f$. Then it holds true
\begin{equation}
\label{e.est0}
\norm{X_F}_{\rm{op}} \leq \ncamp{X_F}_1\ .
\end{equation}
\end{proposition}

We close this Section with a statement (whose proof is also in the
Appendix, see~\ref{aa:Hamfield}) providing the estimate on the
Hamiltonian vector field of a function of class $\Dscr(C_f,\sigma)$.

\begin{lemma}
\label{lem.Hamfield}
Let $F$ be cyclically symmetric homogeneous polynomials of degree $r$
and let its seed $f$ be of class $\Dscr(C_f,\sigma)$; then
\begin{equation}
\label{e.Hamfield}
\ncamp{X_F}_R\leq 4r R^{r-1}\frac{C_f}{(1-e^{-\sigma})^2}\ .
\end{equation}
\end{lemma}

%
%

\section{Results}
\label{s:results}

In this section we present the extensive resonant normal form Theorem
for the Hamiltonian~\eqref{e.H}; in the subsequent
Section~\ref{s:application} we will add some preliminary
applications\footnote{In a forthcoming paper~\cite{PP14b} we will
  exploit further the present theorem for some Breathers stability
  result.} of such a result, exhibiting a generalized dNLS as a first
order normal form of~\eqref{e.H}.

In order to present the result we split the Hamiltonian~\eqref{e.H} as
a sum of its quadratic and quartic parts $H=H_0+H_1$, where
\begin{equation}
\label{e.H.dec}
 H_0(x,y) := \frac12\sum_{j=1}^N \quadr{y^2_j + x^2_j +
   a(x_j-x_{j-1})^2} \ , \qquad H_1(x,y) := \frac14\sum_{j=1}^N x_j^4
 \ .
\end{equation}

\subsection{Normalization of the quadratic part}
\label{ss:norm.quadr}

The first step is the application of the same initial linear
transformation used in~\cite{GioPP12,GioPP13} to give the quadratic
part a resonant normal form. This is a preliminary operation which is
absolutely necessary in order to ``prepare'' the Hamiltonian $H$ for
the forthcoming perturbation algorithm. As widely discussed in the
above cited papers, such a normalization can be implemented using
different approaches.  We recall here a simplified statement of the
  corresponding one of~\cite{GioPP13}. Let us recall the matrix $A$
  introduced in~\eqref{e.H0}
\begin{equation}
\label{e.def-A}
A= (1+2a)\quadr{\Id - {\mu}(\tau + \tau^{\top})} \ ,
\qquad\text{with}\qquad
{\mu}:=\frac{a}{1+2a} \ ,
\end{equation}
which is clearly circulant and symmetric (recall $\tau$ is the
permutation matrix generating~\eqref{e.perm}), and gives a finite
range interaction, in the form of a $\mu$ small\footnote{$\mu$ is
  essentially proportional to the natural small coupling $a$, and is
  always less than one half since we consider positive values of $a$.}
perturbation of the identity. We also introduce the constant frequency
$\Omega$ as the average of the square roots of the eigenvalues of $A$
(actually, the frequencies of the linearized oscillations). Let us
introduce the exponent
\begin{equation}
\label{e.sigma.0}
\sigma_0 := -\ln(2\mu)\ ,
\end{equation}
and take any positive $\sigma_1<\sigma_0$. We have

\begin{proposition}[see \cite{GioPP13}]
\label{p.1}
For ${\mu}<1/2$, the canonical linear transformation $q=A^{1/4} x$,
$p=A^{-1/4}y$ gives the Hamiltonian $H_0$ the particular resonant
normal form
\begin{equation}
\label{e.dec.H0}
 H_0 = H_\Omega + Z_0 \ ,
\qquad
 \Poi{H_\Omega}{Z_0}=0
\end{equation}
with $H_\Omega$ and $Z_0$ cyclically symmetric with seeds
\begin{equation*}
 h_\Omega = \frac\Omega2(q_1^2+p_1^2) \ ,
\qquad
 \zeta_0\in\Dscr\bigl(C_{\zeta_0}(a),\sigma_0\bigr) \ ,
\end{equation*}
and transform $H_1$ into a cyclically symmetric function with seed
\begin{equation*}
h_1\in\Dscr\bigl(C_{h_1}(a),\sigma_1\bigr) \ .
\end{equation*}
\end{proposition}

We stress that it is the above linear transformation which introduces
in a natural way, both in $Z_0$ and in $H_1$, the interaction among
all sites, with an exponential decay with respect to their distance.
Differently from the quadratic interaction $Z_0$, the seed $h_1$
cannot\footnote{We mention here that this loss of the exponential
  decay is a consequence of the requirement that the seed $h_1$ has to
  be left aligned. Indeed it is actually possible to keep
  $h_1\in\Dscr(\cdot,\sigma_0)$, but with a different expansion of
  $h_1=\sum_l h_1^{(l)}$: namely if the support $S(h_1^{(l)})$ is not
  left aligned but ``symmetrically aligned'' around the $0$-th site
  (see also Section \ref{s:application}).} preserve the same
exponential decay rate of the linear transformation; however, as
claimed in the above Proposition, it is possible to show that
$h_1\in\Dscr\bigl(C_{h_1}(a),\sigma_1\bigr)$ for any
$\sigma_1<\sigma_0$. We here make the choice
\begin{equation}
\label{e.sigma.1}
\sigma_1 := \frac12\sigma_0\ ,
\end{equation}
in order to explicitly relate $\sigma_1$ to the small natural
parameter $a$ of the model.

\subsection{Normal Form Theorem}
\label{ss:norm.alg}

With the Hamiltonian transformed by means of the above Proposition
into the form
\begin{equation}
\label{e.H.lintrs}
H = H_\Omega + Z_0 + H_1\ ,
\end{equation}
we are now ready to state the main Theorem. We only anticipate that
the idea is to perform, by using the Lie transform algorithm in the
form explained in \cite{Gio03}, $r$ normalizing steps, provided
$r<r_*(\mu)$. As expected, the maximum number $r_*(\mu)$ of steps
allowed increases when $\mu$ decreases. Moreover, given $\mu$ and $r$,
the normalizing canonical transformation is well defined in a (small)
neighborhood $B_R$ of the origin, where $R<R_*(r,\mu)$. Although this
canonical transformation preserves the extensive nature of the system,
at any step one has to lose a bit of the exponential decay of the
interactions involved in the Hamiltonian.

\begin{theorem}
\label{prop.gen}
Consider the Hamiltonian $H=h^{\oplus}_{\Omega}+\zeta^{\oplus}_0 +
h^{\oplus}_1$ with seeds $h_{\Omega}=\frac{\Omega}{2}(x_0^2+y_0^2)$,
the quadratic term $\zeta_0$ of class $\Dscr(C_{\zeta_0},\sigma_0)$ with
$\zeta_0^{(0)}=0$, and the quartic term $h_1$ of class
$\Dscr(C_{h_1},\sigma_1\,)$. Pick a positive
$\sigma_*\in[\max(\ln(4),\sigma_0/4),\sigma_1)$; then there exist
positive $\gamma$, $\mu_{*}$ and $C_*$ such that for any positive
integer $r$ satisfying
\begin{equation}
\label{e.muperr}
r\lt\frac12\tond{\frac{\mu_*}{\mu}}\ ,
\end{equation}
there exists a finite generating sequence
$\Chi=\{\chi^{\oplus}_1,\ldots,\chi^{\oplus}_r\}$ of a Lie transform
such that $T_{\Chi}\Ham{r}{} = H$ where $\Ham{r}{}$ is a cyclically
symmetric function of the form
\begin{equation}
\label{e.Ham.r}
\Ham{r}{} = H_\Omega + \Z + {P^{(r+1)}}\ ,
\qquad\qquad
\begin{aligned}
\Z :&= Z_0 + \dots + Z_r
\\
\lie{\Omega}Z_s&=0 \ , \quad \forall s\in\{0,\ldots,r\} \ ,
\end{aligned}
\end{equation}
with $Z_s$ of degree $2s+2$ and $P^{(r+1)}$ a remainder starting with
terms of degree equal or bigger than $2r+4$.

Moreover, defining $C_r := 64r^2C_*$ and $\sigma_j := \sigma_1
-\frac{j-1}{r}(\sigma_1-\sigma_*)$, the following statements hold true:
\begin{enumerate}[label=(\roman{*}), ref=(\roman{*})]

\item the seed $\chiph_s$ of $\Chi_s$ is of class
$\Dscr(C_r^{s-1} \frac{C_{h_1}}{\gamma s}, \sigma_s)$.

\item the seed $\zeta_s$ of $Z_s$ is of class
$\Dscr(C_r^{s-1}\frac{C_{h_1}}{s}, \sigma_s)$.

\item with the choice $\sigma_* = \sigma_0/4$, if the
  smallness condition on the energy\footnote{Since $R$ is the radius
    of the ball around the origin considered, the smallness condition
    is in total or specific energy depending on the phase space
    considered, i.e. respectively $\Ph_2$ or $\Ph_\infty$.}
\begin{equation}
\label{e.R.sm1}
R^2<R_*^2:= \frac2{3(1+e)C_r}\ ,
\end{equation}
is satisfied, then the generating sequence $\Chi$ defines an analytic canonical
transformation on the domain $B_{\frac23 R}$ with the properties
\begin{displaymath}
B_{R/3}\subset T_\Chi B_{\frac23 R} \subset B_R\qquad\qquad B_{R/3}\subset
T_\Chi^{-1} B_{\frac23 R} \subset B_R\ .
\end{displaymath}
Moreover, the deformation of the domain $B_{\frac23 R}$ is controlled
by
\begin{equation}
\label{e.def.Tchi}
z\in B_{\frac23 R}\qquad\Rightarrow\qquad \norm{T_\Chi(z)-z}\leq 4^4
C_* R^3\ ,\qquad \norm{T^{-1}_\Chi(z)-z}\leq 4^4
C_* R^3\ .
\end{equation}

\item with the choice $\sigma_*=\sigma_0/4$, if \eqref{e.R.sm1} is
  satisfied, then the remainder is an analytic function on $B_{\frac23
    R}$, and it is represented by a series of cyclically symmetric
  homogeneous polynomials $\Ham{r}{s}$ of degree $2s+2$
\begin{equation}
\label{e.rem.r}
P^{(r+1)} = \sum_{s\geq r+1}\Ham{r}{s}\qquad \Ham{r}{s}
= \tond{h^{(r)}_s}^{\oplus}\ ,
\end{equation}
and the seeds $h^{(r)}_s$ are of class $\Dscr(2\tilde
C_r^{s-1}C_{h_1},\sigma_*)$ with $\tilde C_r = 96 r^2 C_*$.

\end{enumerate}
\end{theorem}

\subsection{Some remarks}
Some comments are in order. First and foremost we stress that our
normal form Theorem holds both in a regime of small total energy and
in a small \emph{specific energy} regime. This fact is somewhat
transparent in the Theorem's statement because the formulation is
given in terms of small neighborhoods of the origin, the radius $R$
being the small parameter: depending on the choice of the norm,
euclidean or supremum one, the control is in total, respectively
specific energy. From the technical viewpoint, this flexibility is
embedded in Proposition~\ref{p.field} which is true both in $\Ph_2$
and in $\Ph_\infty$. From the point of view of the relevance of the
result, the control with specific energy regime, joint with the
uniformity in the number of degrees of freedom, give the validity of
the normal form in the thermodynamic limit.

Clearly the validity of a normal form is only a first step: to fully
exploit it, one has to give some precise control of the dynamics to
ensure that, given suitable conditions on the initial datum, its
evolution remains within the small neighborhood of the origin where
the normal form holds. At present we are able to give such a control
only in a regime of small \emph{total energy}; indeed, in that case, a
rather easy consequence of the normal form Theorem \ref{prop.gen} is
the almost invariance of $H_\Omega$ and $\Z$:

\begin{corollary}
\label{c.Hom.K.var}
Let $z(0)\in B_{\frac19R}$. There exists a constant $C$, independent
of the main parameters $R$ and $a$, such that the approximate
integrals of motion $H_\Omega$ and $\Z$ fulfill
\begin{align*}
|H_\Omega(z(t))-H_\Omega(z(0))| &\leq \Omega R^4\ ,
\\
|\Z(z(t))-\Z(z(0))| &\leq  R^4(C_{\zeta_0} \mu + C_{h_1} R^2)\ ,
\end{align*}
for times
\begin{equation}
  \label{e.times}
  |t|\leq \frac{C (1-e^{-\sigma_*})^2}{C_{h_1}}
              \tond{R^2 C_r}^{-r}\ .
\end{equation}
\end{corollary}

In the time scale of the above Corollary, which is actually of the
order $(Rr)^{-2r}$, one can think of the order $r$ fixed, possibly at
its maximal value of order $1/a$ according to \eqref{e.muperr}, and
then play with the small radius $R$, also provided it is satisfied the
control $R\lesssim 1/r$ given by \eqref{e.R.sm1}.

Actually the Corollary holds because $H_\Omega$ is equivalent to the
euclidean norm, so that its conservation for long times is
self-consistent: it comes from the structure of the normal form, and
at the same time is enough to control the permanence in the right
neighborhood of the origin. Unfortunately the control of the
euclidean norm does not give a control of the sup norm.

In the small specific energy regime, in fact, we are still not able to
exclude that an initial datum with the energy spread all over the
chain could evolve into a localized state for which the sup norm would
grow in a way essentially proportional to the number of degrees of
freedom. And probably this could not be excluded at all. The results
one could hope for, and which we are working on, are the following:
either to show that such localization process takes a very long time,
or that it happens for a set of initial data of small measure (both
things asymptotically with the small parameter given by the specific
energy).
 
Another kind of comments is related to the dependence of the smallness
threshold $R_*$ (defined in \eqref{e.R.sm1}) on the two different
parameters involved in the perturbation construction: the coupling
$\mu$ and the number of iteration steps $r$. We have:
\begin{itemize}
\item at fixed $\mu$, $R_*$ is monotonically decreasing with $r$ (with
  a zero limit if one would be allowed to arbitrarily increase the
  number of steps $r$; recall~\eqref{e.muperr});
\item at fixed $r\geq 1$, $R_*$ increases when decreasing the coupling
  $\mu$ and it has an upper bound independent on $\mu$.
\end{itemize}
One has to observe that, if we remove the coupling from the very
beginning, i.e. $\mu\equiv0$, the system is trivially composed of $N$
identical anharmonic oscillators, and in such a case, it is known that
the Birkhoff normal form procedure is defined on a ball of radius
$0<R_{**}(r)<1$. Indeed our construction reduces to the standard one,
once $\mu$ is set to zero, but our $R_*(r,\mu)$ does not
converge\footnote{Actually, since we have an upper bound on the number
  of steps $r$ whenever $\mu\neq0$, we were not interested in the
  optimization of all the estimates when $r\to\infty$.} to $R_{**}(r)$
as $\mu\to0$.

As a last comment we compare the present result with those of our
previous works \cite{GioPP12,GioPP13}.  There we constructed an
(almost) conserved quantity, here we produce a normal form, which can
give, in principle, much more information about the dynamics of the
system. As a matter of fact, the application we sketch in the above
Corollary~\ref{c.Hom.K.var} resemble very closely the results of the
previous papers, with the following differences: here there's no need
to exclude a small (with the Gibbs measure) set of initial data, but
the result is valid only in total energy. In this sense it is somewhat
complementary. But the above Corollary is only one of the possible
applications once we have a normal form, which can shed more light on
the structure of the Hamiltonian of the system. In the next Section we
start to extract some information looking explicitly at the first step
normal form, which turns out to be a generalized dNLS. We defer a
deeper investigation in such a direction to forthcoming papers. We
plan to explore possible applications of such a normal form: for
example to the stability of Breathers (like in~\cite{PP14b}) and
MultiBreathers, or in order to give a justification for the otherwise
formal use of the (G)dNLS to approximate the evolution of the KG model
with small amplitude initial data. Moreover such a construction could
be extended to the case of interactions, both linear and nonlinear,
beyond the nearest neighbor: the scheme would be exactly the same, the
first step being the study of the decay properties of the linear
transformation (see Proposition~\ref{p.1}), and the second one the
control of the decay loss in the solution of the homological
equation.

%
%
%
%

\section{GdNLS model as normal form for the KG dynamics}
\label{s:application}

Once the Hamiltonian is in the form \eqref{e.H.lintrs}, hence after
the quadratic normalization, if we perform only one step of the
perturbation scheme developed in Theorem~\ref{prop.gen}, i.e. we
choose $r=1$ in \eqref{e.Ham.r}, the transformed Hamiltonian reads
\begin{equation}
  \label{e.K}
  \Ham{1}{} = K + {P^{(2)}} \ ,
\qquad\qquad 
  K := H_\Omega + Z_0 + Z_1\ ,
\qquad\qquad
  {P^{(2)}} = \sum_{s\geq 2}\Ham{1}{s} \ ,
\end{equation}
and the corresponding Hamilton equations are
\begin{equation}
\label{e.Ham.eq}
\dot z = X_K (z) + X_{P^{(2)}}(z)\ .
\end{equation}

\subsection{The Generalized discrete Non Linear Schroedinger
  equation}
\label{ss:GdNLS}

In this part we want to stress and comment the fact that the
simplified Hamiltonian $K$ looks naturally as the Hamiltonian of a
Generalized discrete Non Linear Schroedinger equation (GdNLS). With
the term \emph{generalized} we mean that it includes interactions
among sites which are also beyond the nearest-neighbors, both in the
linear ($Z_0$) and in the nonlinear ($Z_1$) term.

We have indeed, by the normal form construction, the usual additional
conserved quantity given by the $\ell^2$ norm $H_\Omega$
\begin{equation}
  \label{e.gdnls}
  K = H_\Omega + \Z \ ,
  \qquad\qquad
  \Z = Z_0 + Z_1 \ ,
  \qquad\qquad
  \Poi{H_\Omega}{\Z}=0 \ .
\end{equation}
Moreover, due to the decay property of the coefficients of such
interactions, the Hamiltonian $K$ turns out to be a perturbation of
the dNLS model (here presented in real coordinates,
see~\eqref{4.effDNLS} for the standard one in complex coordinates)
\begin{equation}
\label{e.H.dnls}
H_{dNLS} = \sum_{j}\quadr{\frac\Omega2(q_j^2+p_j^2) +
\mu \frac12(q_jq_{j+1}+p_jp_{j+1}) + \frac32 (q_j^2+p_j^2)^2}\ ,
\end{equation}
which is known to be a leading order normal form of the KG
Hamiltonian, when the amplitude is taken proportional to $\sqrt{\mu}$,
which means in the regime $E\sim \mu$, ad discussed for example in the
introduction of \cite{PelS12} (see also~\cite{DauDP97,ClaKKS93} for
other examples of the use of the dNLS in the study of a KG model).

Indeed, both the seeds $\zeta_0$ and $\zeta_1$ of the quadratic and
quartic terms $Z_0$ and $Z_1$ include interactions which are
exponentially small with the distance among sites, with the following
expansions:
\begin{equation}
\label{e.z0z1.shape}
\zeta_0 = \sum_{m=1}^{[N/2]}\zeta_0^{(m)} \ ,
\qquad\qquad
\zeta_1 = \sum_{m=0} ^{[N/2]}\zeta_1^{(m)} \ ,
\end{equation}
with supports for the components $\zeta_j^{(m)}$
\begin{equation*}
S(\zeta_j^{(m)})\subset [0,\ldots,m]\cup[N-m,\ldots,N] \ .
\end{equation*}
For the quadratic part we have an explicit expression:
\begin{equation}
\label{e.Z0.shape}
\zeta_0^{(m)} = b_m\quadr{q_0(q_m+q_{N-m})+p_0(p_m+p_{N-m})}\ ,
\qquad\qquad
|b_m| = \mathcal{O}\tond{e^{-\sigma_0 m}} \ ,
\end{equation}
while for the quartic one we present here only a control of the norm
of the components
\begin{equation*}
\norm{\zeta_1^{(m)}}\leq C_{h_1}' e^{-\sigma_0 m} = C_{h_1}' (2\mu)^m \ .
\end{equation*}
The effective computations of the monomials included in all the
$\zeta_1^{(m)}$ is indeed a doable task, at least if supported by an
algebraic manipulator program. We nevertheless defer such a task to
future developments whenever it will be a necessary step.

With respect to the small parameter $\mu$, the leading terms of
$\zeta_0$ and of $\zeta_1$ are respectively the (resonant)
nearest-neighbors interaction of the dNLS model and its nonlinear part
(see, for comparison, formula~\eqref{e.H.dnls}, recalling that $b_1$
is $\mathcal{O}(\mu)$)
\begin{displaymath}
\zeta_0 = b_1\quadr{q_0(q_1+q_{N-1})+p_0(p_1+p_{N-1})} +
     {\mathcal{O}(\mu^2)} \qquad\qquad \zeta_1 =
     \frac32(q_0^2+p_0^2)^2 + \mathcal{O}(\mu) \ .
\end{displaymath}

Concerning $\zeta_1$ we remark that, with respect to the expansion
used in the normal form construction, we here\footnote{The same
  approach could be extended elsewhere but it is beyond the purposes
  of the present work.} exploit the previously mentioned idea of
taking the support $S(\zeta_1^{(m)})$ symmetrically centered around
the $0$-th variable. This provides the decay rate $\sigma_0$, which
represents a stronger condition than
$\zeta_1\in\Dscr(C_{h_1},\sigma_1)$, with $\sigma_1<\sigma_0$, claimed
in Proposition~\ref{p.1}. This different expansion is a
straightforward consequence of the following Lemma, whose proof is
deferred to the Appendix
\begin{lemma}
\label{l.H1.sym.exp}
It is possible to select a seed $h_1$ such that 
\begin{equation}
\label{e.h1.sym.seed}
h_1 = \sum_{m=0} ^{[N/2]}h_1^{(m)}\ ,\qquad S\tond{h_1^{(m)}} \subset
[0,\ldots,m]\cup[N-m,\ldots,N]\ .
\end{equation}
Moreover, there exists $C_{h_1}'$ such that
\begin{displaymath}
\norm{h_1^{(m)}}\leq C_{h_1}' e^{-m\sigma_0}\ .
\end{displaymath}
\end{lemma}

\subsection{Further comments on the construction of the normal form
  and on its relationship with the GdNLS model.}
\label{ss:comments}

Here we aim at giving some comments on our results based on the
remarks that we have two natural small parameters in our model,
i.e. the coupling $a$ and the energy $E$.

In the framework of a perturbation construction, at least at the
formal level, the presence of a small parameter is usually exploited,
by an expansion in its powers, to give a natural ordering of the terms
which are dealt with at every step of the iterative procedure. If two
small parameters are involved, we face a problem of gradation, which
clearly comes from the lack of a natural ordering in $\ZZ^2$. This is
usually dealt with by choosing in advance a particular relation
between the two parameters, i.e. by fixing a particular ``regime'',
thus effectively reducing the model to a system with a single
perturbation parameter. For example, as already remarked, it is well
known that by setting $E\sim a$ the standard dNLS arises as the first
order normal form for the KG model \eqref{e.H} (see \cite{PelS12}).
In such a case, it is ensured that, at least formally, at every step
of the normal form construction, the remainder is smaller than the
normal form terms.

Some comments are in order. The first is that for every different
regime one aims to consider, a different normal form arises, in
particular if higher order terms are involved: thus it is necessary to
fix the ratio between the parameters in advance. Moreover, the
procedure becomes rapidly quite cumbersome at the level of the
selection of the terms which have to be dealt with at every step. An
example is given in the next subsection~\ref{a.classic}, where the KG
model is dealt with as described above, and with the choice of
$a\lesssim E$, the dNLS appears after two steps.

In presence of more than one small parameter, even if one aims at
producing a unique normal form to be used in different regimes of the
parameters, the ordering used to treat the terms during the procedure
must be chosen a priori, using some suitable criterion. In the present
work, we deal with all orders in $a$ -- actually in a single step --
for every fixed order in $E$. Indeed, every term $Z_j$
in~\eqref{e.Ham.r} contains the corrections for all the powers in $a$,
and the index $j$ relates to the order in $E$. This makes sense, and
it is doable, exactly because the normalization with respect to $a$
converges. The price one has to pay is that of course every term $Z_j$
in~\eqref{e.Ham.r} keeps also contributions which are smaller than
those contained in the remainder: indeed, no matter how small $E$ is,
in every $Z_j$ we have contributions containing $a^l$ with $l$
arbitrarily large, so that $E^j a^l < E^r$. The particular regime
taken into account will determine how many terms, and which of them,
are in this situation.


\subsection{Comparison with a ``standard'' normalizing approach}
\label{a.classic}

Let us apply to the original Hamiltonian~\eqref{e.H} the scaling
$X_j = x_j/\sqrt{E}\ , Y_j = y_j/\sqrt{E}$; we obtain
\begin{equation}
  \label{4.KGresc}
  H(X,Y,a,E) = \frac{1}{E}H(x(X),y(Y)) =
  \sum_{j=0}^{N-1} \left[
    {\frac{X_j^2+Y_j^2}2} +
    a{\frac{(X_{j+1}-X_j)^2}2} +
    E{\frac{X_j^4}4}
  \right] \ ;
\end{equation}
the system actually presents two effective perturbation parameters,
which are independent: $E$ and $a$. In the limit $E\ll a< 1$ the
dynamics is essentially governed by the whole quadratic part and the
linear approximation prevails over the nonlinear effects. In the
complementary case, $a\lesssim E< 1$, the on-site nonlinear dynamics
is at least as relevant as the small coupling among nearby sites.

We are going to show that, if $a\lesssim E<1$, then a ``standard''
normalizing procedure gives the dNLS as a resonant normal form of the
KG model. Moreover, it essentially coincides\footnote{The only
  difference will be the coefficient of the next-neighbors coupling,
  since $b_1\not=a/2$, although $b_1=\mathcal{O}(a)$.}  with the
leading part of the normal form \eqref{e.H} discussed in
Section~\ref{s:application}.

\paragraph{Birkhoff complex coordinates:}

We put the Hamiltonian into Birkhoff complex coordinates $\xi_j =
(X_j+ \Im Y_j)/{\sqrt{2}}\ , \Im \eta_j = \overline{\xi}_j$, so that
\eqref{4.KGresc} reads
\begin{equation}
\label{1.KGbirkh}
H(\xi,\eta) = h_\omega(\xi,\eta) + f^{(0)}(\xi,\eta,a) +
\Ham{0}{1}(\xi,\eta,E),
\end{equation}
where the norm has been normalized to $h_\omega=1$ and the
``perturbation'' is composed of
\begin{align*}
f^{(0)}(\xi,\eta,a) &=
\frac{a}4\sum_{j=0}^{N-1}{\left(\xi_{j+1}^2+\xi_j^2-\eta_{j+1}^2-\eta_j^2
  -2\xi_{j+1}\xi_j+2\eta_{j+1}\eta_j\right)} + \\ &+
\frac{a}2\sum_{j=0}^{N-1}{\left(\xi_{j+1}-
  \xi_{j}\right)\left(\overline\xi_{j+1}-
  \overline\xi_{j}\right)},\\ \Ham{0}{1}(\xi,\eta,E) &=
\frac{E}{16}\sum_{j=1}^N{\left(\xi_{j}^4 + \eta_{j}^4 + 4
  \xi_j^2|\xi_j|^2 - 4\eta_j^2|\xi_j|^2 + 6|\xi_j|^4\right)}\ .
\end{align*}

We perform a resonant normal form construction with respect to the
resonant module
\begin{displaymath}
\mathcal{M}_{\omega}:=\{ k\in\mathbb{Z}^N | \langle k,\omega\rangle =0\} = \{
k\in\mathbb{Z}^N | \, k_1 + k_2 + \ldots + k_N=0\}.
\end{displaymath}

\paragraph{First step:} the first term $f^{(0)}$ is already
split into a $\Rscr^2$ and a $\Nscr^{ 2}$ part: hence it is
possible to define
\begin{equation}
\label{2.Z0}
Z_0^{(0)}(\xi,\eta) = \frac{a}2\sum_{j=1}^N{|\xi_{j+1}-\xi_j|^2} =
\Pi_{\Nscr^2} f^{(0)}\ ,\qquad \Poi{h_{\omega}}{Z_0^{(0)}}=0\ ,
\end{equation}
as the $a$-resonant term and remove the Range part via a generating
function $\Chi_0=\mathcal{O}(a)$ which satisfies the usual homological
equation
\begin{equation}
\label{e.om.chi0}
\Poi{\Chi_0}{h_{\omega}} = Z_0^{(0)} - f^{(0)}\ ,\qquad \Chi_0:=
\lie{h_\omega}^{-1}\quadr{f^{(0)} - Z_0^{(0)}}\ .
\end{equation}
The change of coordinates $\T_{\Chi_0}$ gives the Hamiltonian the
shape
\begin{equation}
\label{2.KG1step}
H(\xi,\eta) = \sum_{j=1}^N{|\xi_j^2|} + \frac{a}2
\sum_{j=1}^N{|\xi_{j+1}-\xi_j|^2} + \Ham{0}{1}(\xi,\eta) +
\mathrm{h.o.t.}\ ,
\end{equation}
where we have still used $\xi,\eta$ to indicate the new
coordinates. The higher order terms are
\begin{equation}
\textrm{h.o.t} = \quadr{\T_{\Chi_0}h_\omega - h_\omega -
  \lie{\Chi_0}h_\omega} + \quadr{\T_{\Chi_0}f^{(0)} - f^{(0)}} +
\quadr{\T_{\Chi_0}\Ham{0}{1} - \Ham{0}{1}}\ ,
\end{equation}
where the first two are of order $\mathcal{O}(a^2)$, while the last is
of order $\mathcal{O}(aE)$.

The previous analysis forces to compare the newly generated quadratic
term $f^{(1)} := \T_{\Chi_0}f^{(0)} - f^{(0)}$ with the quartic
potential $\Ham{0}{1}$, the first being $\mathcal{O}(a^2)$ and the
second $\mathcal{O}(E)$. The part of the remainder which is
$\mathcal{O}(aE)$ can be neglected at this step, being much smaller
than $\Ham{0}{1}$. If $a\lesssim E$ then $f^{(1)}$ can be transferred
in the remainder and we can pass to consider $\Ham{0}{1}$ as the next
term to be normalized. 

\paragraph{Second step with ${\bf a\lesssim E}$:} The resonant term in
$\Ham{0}{1}$ reads
\begin{equation}
\label{3.Z2}
Z_1^{(0)}= \frac{3E}8 \sum_{j=0}^{N-1}{|\xi_j|^4}.
\end{equation}
Thus removing through $\Chi_1$ all the Range terms of $\Ham{0}{1}$, we
obtain the normal form
\begin{equation}
\label{3.KG2step}
H(\xi,\eta)= \overline{K} + \mathrm{h.o.t.}\ ,
\qquad
\overline{K}:= h_\omega+Z_0^{(0)} + Z_1^{(0)}\ ,
\qquad
\Poi{h_\omega}{\overline{K}}=0 \ .
\end{equation}
Such a normal form of $H$ is a dNLS model
\begin{equation}
\label{4.effDNLS}
\overline{K}(\xi,\eta) = \sum_{j=0}^{N-1}{|\xi_j^2|} + \frac{a}2
\sum_{j=0}^{N-1}{|\xi_{j+1}-\xi_j|^2} +
\frac{3E}8\sum_{j=0}^{N-1}{|\xi_j|^4}\ .
\end{equation}

\noindent

%
%

\section{Proof of Theorem \ref{prop.gen} and of Corollary \ref{c.Hom.K.var}}
\label{s:construction}

This section is devoted to the proof of Theorem~\ref{prop.gen} and of
its immediate Corollary~\ref{c.Hom.K.var}.  In
Section~\ref{ss:formal.tchi} we include a formal part, where we recall
the process of construction of the normal form and discuss the
solvability of the homological equation. In Section~\ref{ss:estimates}
we give all the quantitative estimates yielding to the statements of
Theorem \ref{prop.gen}. Finally in Section~4.4 we prove Corollary
\ref{c.Hom.K.var}.

%
%

\subsection{{Previous results: formal algorithm and solution of the homological equation}}
\label{ss:formal.tchi}

We recall here some basic facts on the formal algorithm we use to
construct the normal form and to estimate its remainder: we refer to
\cite{Gio03} for a detailed treatment.

Given a truncated generating sequence $\Chi=\{\Chi_s\}_{s=1,\dots,r}$,
we define the linear operator $T_\Chi$ as
\begin{equation}
\label{e.TelChi}
T_\Chi=\sum_{s\geq 0}E_s,\qquad E_0=\Id,\qquad E_s= \sum_{j=1}^s
\tond{\frac{j}{s}}\lie{\Chi_j}E_{s-j}\ ,
\end{equation}
where $\lie{\Chi_j}\cdot=\Poi{\Chi_j}{\cdot}$ is the Lie derivative
with respect to the flow generated by $\Chi_j$.  

We look for $\Chi$ and a function $\Ham{r}{}$ which represents a
resonant \emph{normal form} for the original Hamiltonian $H$, which
means that $\Ham{r}{}$ satisfies the equation $T_\Chi \Ham{r}{}=H$,
with $\Ham{r}{}$ of the form \eqref{e.Ham.r}, where the normalized
terms are in normal form in the sense that
\begin{displaymath}
\Poi{H_\Omega}{Z_s}=0\qquad \forall s=0,\ldots,r\ .
\end{displaymath}
An immediate consequence is that $H_\Omega$ is an approximated first
integral for the transformed Hamiltonian $\Ham{r}{}$, since
\begin{displaymath}
\Poi{H_\Omega}{\Ham{r}{}} = \Poi{H_\Omega}{P^{(r+1)}}\ .
\end{displaymath}

We now translate the equation $T_\Chi Z=H$ into a formal recursive
algorithm that allows us to construct both $Z$ and $\Chi$. We take
into account that our Hamiltonian has the particular form
$H=H_0+H_1\,$, where $H_1$ is a homogeneous polynomial of degree 4.

For  $s\geq 1$ the generating function $\Chi_{s}$ and the normalized
term $Z_s$ must satisfy the recursive set of {\sl homolo\-gical equations}
\begin{equation}
\label{e.om.2s}
\lie{H_0}\Chi_{s} = Z_{s} + \Psi_{s};
\end{equation}
where 
\begin{align}
\label{e.Psi.2s}
\Psi_{1} &= H_1,\\ \Psi_{s} &= \frac{s-1}{s}\lie{\Chi_{s-1}}H_1 +
\sum_{j=1}^{s-1} \frac{j}{s} E_{s-j}Z_{j}\ ,\qquad s\geq2\ .
\end{align}

Our aim is to solve the homological equation \eqref{e.om.2s} with the
prescription that $\lie{\Omega}Z_s=0$ where $\lie{\Omega}\cdot
= \{H_{\Omega},\cdot\}$ is the Lie derivative along the vector field
generated by $H_{\Omega}$ as defined in \eqref{e.dec.H0}.  Thus we
first point out the properties of the operator $\lie{\Omega}$, and
then discuss how to solve the homological equation using a Neumann
series expansion of the operator $\lie{H_0}$, which is a $\mu$
perturbation of $\lie{\Omega}$.

\subsubsection{The linear operator $\lie{\Omega}$}
\label{ssect.2.1}

It is an easy matter to check that $\lie{\Omega}$ maps the space of
homogeneous polynomials into itself.  It is also well known that
$\lie{\Omega}$ may be diagonalized via the canonical transformation
\begin{equation}
x_j = \frac{1}{\sqrt{2}}(\xi_j+i\eta_j)\ ,\quad
y_j = \frac{i}{\sqrt{2}}(\xi_j-i\eta_j)\ ,\quad
j=1,\ldots,N\ ,
\label{pertur.51}
\end{equation}
where $(\xi,\eta)\in\CC^{2n}$ are complex variables.  For a
straightforward calculation gives
\begin{equation}
\lie{\Omega}\xi^j\eta^k =
i\Omega\,(|k|-|j|)\,\xi^j\eta^k\ ,
\label{pertur.80}
\end{equation}
where $|j|=|j_1|+\ldots+|j_N|$ and similarly for $|k|\,$.\footnote{ If
$f(x,y)=\sum_{j,k} c_{j,k} x^jy^k$ is a real polynomial, then
\eqref{pertur.51} produces a polynomial $g(\xi,\eta)=\sum_{j,k}
b_{j,k}\xi^j\eta^k$ with complex coefficients $b_{j,k}$ satisfying
$b_{j,k} = - b_{k,j}^*\ ,$
and conversely.}

Let us denote by $\Pscr^{(s)}$ the linear space of the homogeneous
polynomials of degree $s$ in the $2n$ canonical variables
$\xi_1,\ldots,\xi_n,\eta_1,\ldots,\eta_n\,$.  The kernel and the range
of  $\lie{\Omega}$ are defined in the usual way, namely
\begin{equation*}
\Nscr^{(s)} = \lie{\Omega}^{-1}(0)\ ,\qquad
\Rscr^{(s)} = \lie{\Omega}(\Pscr^{(s)})
\end{equation*}
The property of $\lie{\Omega}$ of being diagonal implies
\begin{equation*}
\Nscr^{(s)}\cap\Rscr^{(s)} = \{0\}\ ,\qquad
\Nscr^{(s)}\oplus\Rscr^{(s)} = \Pscr^{(s)}\ .
\end{equation*}
Thus the inverse $\lie{\Omega}^{-1}:\range\rightarrow\range$ is
uniquely defined.  


\subsubsection{The linear operator $\lie{H_0}$}
\label{ssect.2.2}
We come now to the solution of the homological
equation~(\ref{e.om.2s}).  In view of~\eqref{e.dec.H0} we have
$\lie{H_0} = \lie{\Omega} + \lie{Z_0}$, or equivalently $\lie{H_0} =
\lie{\Omega}\tond{\Id + \lie{\Omega}^{-1}\lie{Z_0}}$.  Thus we have
\begin{equation}
\label{e.inv.lieH0}
\lie{H_0}^{-1} = \tond{\Id + K}^{-1}\lie{\Omega}^{-1}\ ,
\qquad\qquad
K := \lie{\Omega}^{-1}\lie{Z_0}\ ,
\end{equation}
and using the Neumann's series we formally get $\tond{\Id + K}^{-1} =
\sum_{l\geq 0}(-1)^l K^l$.

A general consideration is the following.  Let us consider $\lie{H_0}$
on a topological space $\pspazio$ endowed with any norm
$\norm{\cdot}$. The next Proposition (see
again~\cite{GioPP13},~Section~4) claims that, although we ignore its
Kernel and Range, we can invert $\lie{H_0}$ on $\range$.

\begin{proposition}[see \cite{GioPP13}]
\label{p.2.1}
If the restriction of $K$ to $\range$ satisfies
\begin{equation}
\label{e.norm.K}
\norm{K}_{\rm op}<1,
\end{equation}
then for any $g\in\range$, there exists a unique element $f\in\range$
of the form
\begin{displaymath}
f = \sum_{l\geq 0}(-1)^l K^l g \ ,
\qquad\text{such that}\qquad
\tond{\Id+K}f = g \ .
\end{displaymath}
\end{proposition}

%
%

\subsection{Proof of Theorem \ref{prop.gen}}
\label{ss:estimates}

In order to show that the exponential decay of interactions is
preserved by our construction, our first aim (as in \cite{GioPP13}) is
to show that the functions $\Chi_s$, $\Psi_s$ and $Z_s$, that are
generated by the formal construction, are of class
$\Dscr(\cdot,\sigma_s)$, with suitable values of $\sigma_s$ and with
some constant to be evaluated in place of the dot. A second step
pertains the estimate of the remainder $\R^{(r+1)}$, which has still
to be expressed in terms of $\Chi_s$, $\Psi_s$ and $Z_s$. We conclude
with the estimates of the Hamiltonian vector fields of the generating
functions $\Chi_s$, which allow to control the deformation of the
transformation $\T_\Chi$.

Let us pick $\sigma_*\lt\sigma_1$, and recall we have defined in
Theorem \ref{prop.gen}
\begin{equation}
\label{succ.sigma}
\sigma_j := \sigma_1 - \frac{(j-1)}{r}(\sigma_1-\sigma_*)\ ,
\quad{\rm for}\ j=1,\ldots,r\ ,
\end{equation}
so that $\sigma_1\gt\ldots\gt\sigma_r>\sigma_{r+1}=\sigma_*\,$.
We shall repeatedly use the following elementary estimates. By the
general inequality
\begin{equation}
\label{e.exp.ineq}
1-e^{-x} \ge \tond{\frac{x}{a}}(1-e^{-a}) \quad{\rm for}\ 0\le x\le a\ ,
\end{equation}
for $0\le j\lt s\le r$ we get
\begin{equation}
\label{eq.denpois}
\vcenter{\openup1\jot\halign{
\hfil$\displaystyle{#}$
&$\displaystyle{#}$\hfil
\cr
1-e^{-\max(\sigma_j,\sigma_{s-j})} 
 &\ge \frac{1 -e^{-\sigma_0}}{\sigma_0}
   \max(\sigma_j,\sigma_{s-j})
 \ge \frac{(1 -e^{-\sigma_0})}{4}\ ,
\cr
1 - e^{-(\sigma_j-\sigma_s)}
 &\ge \frac{s-j}{r}(1-e^{-(\sigma_0-\sigma_*)}) 
\quadr{\frac{\min(\sigma_0-\sigma_1,\sigma_1-\sigma_*)}{\sigma_0-\sigma_*}}\
 .
\cr
}}
\end{equation}
To get the first inequality we make use of
$\max(\sigma_j,\sigma_{s-j}) \geq (\sigma_1+\sigma_*)/2 \geq \sigma_*$ and of the hypothesis 
$\sigma_*\geq\sigma_0/4$.  Concerning the second of \eqref{eq.denpois},
take first $0=j<s\leq r$ and apply \eqref{e.exp.ineq} to get
\begin{align*}
  1 - e^{-(\sigma_0-\sigma_s)}
\geq
  \quadr{\frac{(\sigma_0-\sigma_1) + \frac{s-1}{r}(\sigma_1-\sigma_*)}{\sigma_0-\sigma_*}}
  &(1-e^{-(\sigma_0-\sigma_*)})
>\\>
  \frac{s}{r}\tond{\frac{\sigma_0-\sigma_1}{\sigma_0-\sigma_*}} &(1-e^{-(\sigma_0-\sigma_*)}) \ ;
\end{align*}
then take $1\leq j<s\leq r$ and apply \eqref{e.exp.ineq} to get
\begin{displaymath}
  1 - e^{-(\sigma_j-\sigma_s)} \geq 
  \tond{\frac{\sigma_j-\sigma_s}{\sigma_0-\sigma_*}}(1-e^{-(\sigma_0-\sigma_*)})
  = \frac{s-j}{r}\tond{\frac{\sigma_1-\sigma_*}{\sigma_0-\sigma_*}}
  (1-e^{-(\sigma_0-\sigma_*)}) \ .
\end{displaymath}

%
%
\paragraph{Estimate of the homological equation.}
\label{par:homeq}
We summarize here the main result and comments which can be found
in \cite{GioPP13}, proof included, about the estimate of the
homological equations \eqref{e.om.2s}.

In order to proceed we first define
\begin{equation}
\label{e.E0*}
E_{0*}:= \frac{\min(\sigma_0-\sigma_1,\sigma_1-\sigma_*)}{\sigma_0-\sigma_*}\ .
\end{equation}
and then we give consistent values for the constants of Theorem~\ref{prop.gen}
\begin{equation}
\label{const.prop.gen}
\begin{aligned}
 \mu_* &= \frac{\Omega(1-e^{-\sigma_0})(1-e^{-(\sigma_0-\sigma_*)})E_{0*}}
               {8C_{\zeta_0}e^{\sigma_1}} \ ,
\\
 \gamma &= 2\Omega\Bigl(1-\frac{r\mu}{\mu_*}\Bigr) \ ,
\\
 C_*  &= \frac{C_{h_1}}{\gamma(1-e^{-\sigma_0})(1-e^{-(\sigma_0-\sigma_*)})E_{0*}} \ .
\end{aligned}
\end{equation}

\begin{lemma}
\label{l.yal}
Let $G=g^{\oplus}\in\Rscr^{(2s+2)}$ be a cyclically symmetric
homogeneous polynomial of degree $2s+2$ of class $\Dscr(C_g,\sigma_s)$.
Let $K$ as defined in~\eqref{e.inv.lieH0} and assume
\begin{equation}
\label{CK}
C_K := \frac{4 C_{0}
e^{-(\sigma_0-\sigma_1)}}{\Omega(1-e^{-\sigma_0})(1-e^{-(\sigma_0-\sigma_*)})E_{0*}}
\le \frac{1}{2r}\ .
\end{equation}
Then there exists a cyclically symmetric homogeneous polynomial $\Xscr =
\chi^{\oplus}\in\Rscr^{(2s+2)}$ which solves $\lie{H_0}\Xscr = G$;
moreover $\chi$ is of class $\Dscr(C_g/\gamma,\sigma_s)$ with
\begin{equation}
\label{e.omol.1}
 \gamma = 2\Omega(1-rC_K)\ .
\end{equation}
\end{lemma}

\begin{remark}
\label{r.picc.div}
In Proposition~\ref{p.2.1} we ask $\norm{K}_{\rm op}<1$ to simply
perform the inversion. In the above Lemma~\ref{l.yal},
condition~\eqref{CK} reads as $\norm{K}_{\rm op}<1/2$, and this stronger
requirement is to control the small divisors~\eqref{e.omol.1}.
\end{remark}

\noindent
We emphasize that in view of the first of~\eqref{const.prop.gen} we
have $C_K=\mu/\mu_*\,$.  Therefore, condition~\eqref{CK} reads
$2r\mu<\mu_*$, which is the smallness condition for $\mu$ of
Theorem~\ref{prop.gen}.  Furthermore this gives the value of $\gamma$
in~\eqref{const.prop.gen}. Moreover, the constant $\gamma$ is
evaluated as independent of $s$, but seems to depend on the degree $r$
of truncation of the first integral.  However, in view of the
condition on $\mu$ we have $\Omega\leq\gamma\leq2\Omega$.

Having thus proved that the homological equation can be solved, the
statement~(i) of Theorem~\ref{prop.gen} follows.

%
%

\paragraph{Iterative estimates on the generating sequence.}
We follow the same procedure used in \cite{GioPP13}, Subsection 4.2.
We recall that the generating sequence is found by recursively solving
the homological equations $\lie{H_0}\chiph_{s} =Z_{s}+\Psi_{s}$ for
$s=1,\ldots,r$ with
\begin{equation}
\label{est.gs.1}
\vcenter{\openup1\jot\halign{
\hfil$\displaystyle{#}$
&$\displaystyle{#}$\hfil
&$\displaystyle{#}$\hfil
\cr
\Psi_1 &= H_1\ ,
\cr
\Psi_{s} &=
 \tond{\frac{s-1}{s}}\lie{\Chi_{s-1}} H_1
  + \sum_{l=1}^{s-1} \tond{\frac{l}{s}} E_{s-l} Z_{l}\ ,
\cr
E_s Z_{l}
&= \sum_{j=1}^{s} \tond{\frac{j}{s}}\lie{\Chi_{j}}E_{s-j}Z_{l}
&\quad{\rm for}\ s\ge 1\ .
\cr
}}
\end{equation}
Our aim is to find positive constants $C_{\psi,1},\ldots,C_{\psi,r}$
so that $\Psi_{s}$ is of class $\Dscr(C_{\psi,s},\sigma_s)$.  In view
of lemma~\ref{l.yal} this implies that $Z_{s}$ of class
$\Dscr(C_{\zeta,s},\sigma_s)$ with $C_{\zeta,s}=C_{\psi,s}$ and
$\chiph_{s}$ of class $\Dscr(C_{\chi,s},\sigma_s)$ with
$C_{\chi,s}=C_{\psi,s}/\gamma\,$.  Meanwhile we also find constants
$C_{\zeta,s,l}$ such that $E_{s} Z_{l}$ is of class
$\Dscr(C_{\zeta,s,l},\sigma_{s+l})$ whenever $s+l\le r$.

We look for a constant $B_r$ and two sequences 
$\{\eta_s\}_{1\le s\le r}$ and $\{\theta_{s}\}_{1\le s\le r}$ such
that 
\begin{equation}
\label{est.gs.5}
\vcenter{\openup1\jot\halign{
\hfil$\displaystyle{#}$
&$\displaystyle{#}$\hfil
&$\displaystyle{#}$\hfil
\cr
C_{\psi,1} 
 &\le \eta_1C_{h_1}\ ,
\quad
C_{\zeta,0,1} \le \eta_1\theta_0 C_{h_1}\ ,
\cr
C_{\psi,s} &\le \frac{\eta_s}{s} {B_r^{s-1}C_{h_1}}
&\quad {\rm for}\ s\gt 1\ ,
\cr
C_{\zeta,s,l} &\le \theta_{s} \eta_l {B_r^{s+l-1}C_{h_1}}
&\quad {\rm for}\ s\ge 1\,,\> l\ge 1\ .
\cr
}}
\end{equation}

In view of $\Psi_1=H_1$ and of $E_0Z_{1}=Z_{1}$ we can choose
$\eta_1=\theta_0=1$.  By~(\ref{est.gs.1}) and using lemmas~\ref{l.yal}
and~\ref{lem.poisson} together with corollary~\ref{cor.poisson} we get
the recursive relations
\begin{equation}
\label{est.gs.2}
\vcenter{\openup1\jot\halign{
\hfil$\displaystyle{#}$
&$\displaystyle{#}$\hfil
\cr
C_{\zeta,s,l} 
&\le \frac{4}{s} \sum_{j=1}^{s-1} 
\frac{j(s+l-j)\eta_j\eta_l\theta_{s-j}}{(1-e^{-\max(\sigma_j,\sigma_{s+l-j})+\sigma_{s+l}})
        (1-e^{-\max(\sigma_j,\sigma_{s+l-j})})} \cdot 
        \frac{B_r^{s+l-2} C_{h_1}^2}{\gamma}\ .
\cr
C_{\psi,s} 
&\le \biggl(
  \frac{8(s-1) \eta_{s-1} C_{h_1}}{s
        (1-e^{-(\sigma_0-\sigma_{s})})(1-e^{-\sigma_0})} 
  + \sum_{l=1}^{s-1} \frac{lB_r}{s} \eta_l\theta_{s-l}
  \biggr) 
 \cdot \frac{B_r^{s-2}C_{h_1}}{\gamma}\ ,
\cr
}}
\end{equation}
Using the first of \eqref{eq.denpois}, with $s+l$ instead of $s$, we have
\begin{displaymath}
1-e^{-\max(\sigma_j,\sigma_{s+l-j})} > \frac{1-e^{-\sigma_0}}{4}\ .
\end{displaymath}
Using the second of~\eqref{eq.denpois} in a similar way to deal with
\begin{displaymath}
1-e^{-[\max(\sigma_j,\sigma_{s+l-j})-\sigma_{s+l}]} \geq 
\frac{s+l-\min(j,s+l-j)}{r}\tond{1-e^{-(\sigma_0-\sigma_*)}}E_{0*}\ ,
\end{displaymath}
and setting
\begin{equation}
\label{est.gs.4}
B_r = \frac{16 C_{h_1} r}{\gamma
  (1-e^{-(\sigma_0-\sigma_*)})(1-e^{-\sigma_0})E_{0*}}\ .
\end{equation}
we get
\begin{equation}
\label{est.gs.3}
\vcenter{\openup1\jot\halign{
\hfil$\displaystyle{#}$
&$\displaystyle{#}$\hfil
\cr
C_{\zeta,l,s} 
&\le \frac{1}{s} \sum_{j=1}^{s-1} 
      {j\eta_j\eta_l\theta_{s-j}}\,
       {B_r^{s+l-1} C_{h_1}}\ ,
\cr
C_{\psi,s} 
&\le \biggl(
      \frac{1}{s}\eta_{s-1} 
      +\sum_{l=1}^{s-1} 
        \frac{l}{s} \eta_l\theta_{s-l}
     \biggr)
  {B_r^{s-1} C_{h_1}}\ .
\cr
}}
\end{equation}
Thus the inequalities~\eqref{est.gs.5} are satisfied by the sequences
recursively defined as
$$
\vcenter{\openup1\jot\halign{
\hfil$\displaystyle{#}$
&$\displaystyle{#}$\hfil
&$\displaystyle{#}$\hfil
\cr
\theta_{s} &:=
\sum_{j=1}^{s-1} \frac{j}{s} \eta_j\theta_{s-j} 
&\quad{\rm for}\ s\ge 1\ ,
\cr
\eta_s &:=
 \eta_{s-1} +\sum_{j=1}^{s-1} j \eta_j\theta_{s-j} &\quad{\rm for}\
      s\ge 2\ .
\cr
}}
$$ starting with $\eta_{1} =\theta_{0} = 1$.  It is possible,
recalling also that $s\leq r$, to show that (see \cite{GioPP13})
\begin{displaymath}
\eta_s  < 4^{s-1} r^{s-1}\ .
\end{displaymath}

Replacing this and~\eqref{est.gs.2} in the inequality~(\ref{est.gs.5})
for $C_{\psi,s}$ and recalling that $C_{\chi,s}\le C_{\psi,s}/\gamma$
we have
$$
C_{\chi,s} \le (64r^2C_*)^{s-1}\frac{C_{h_1}}{\gamma s}\ ,\qquad
 C_* =  \frac{C_{h_1}}{\gamma(1-e^{-\sigma_0})(1-e^{-(\sigma_0-\sigma_*})E_{0*}}\ .
$$
The proves the statement~(ii) of Theorem~\ref{prop.gen} with the
estimated value of $C_*$ in~\eqref{const.prop.gen}.  The
statement~(iii) also follows in view of $C_{\zeta,s}\le C_{\psi,s}\,$.

%
%

\paragraph{The remainder of the normal form.}

We combine here the formal algorithm developed in Chapter 4
of \cite{Gio03} with the previous estimates on
$\chi_s,\,\Psi_s,\,Z_s$. In general, the remainder $P^{(r+1)}$ can be
written using the operators $D_s$
\begin{equation}
\label{e.D.s}
D_s := -\sum_{j=1}^s\tond{\frac{j}{s}} D_{s-j}\lie{\chi_j}\qquad\qquad
D_0 := \Id\ ,
\end{equation}
which define the inverse of the canonical transformation generated by
$\T_\chi$
\begin{displaymath}
\T_\chi^{-1} := \sum_{s\geq 0}D_s\ .
\end{displaymath}
Thus, as stressed in Paragraph 5.2.3 of \cite{Gio03}, one has
\begin{equation}
\label{e.R.r}
P^{(r+1)} = \sum_{s>r}{\Ham{r}{s}}\ ,\qquad\qquad \Ham{r}{s}
:= \sum_{j=0}^{s-1}D_jH_{s-j} - \sum_{j=1}^s\frac{j}{s}
D_{s-j}\tond{\Psi_j-Z_j}\ .
\end{equation}
In our specific case
\begin{displaymath}
\Ham{r}{s} = D_{s-1}H_1- \sum_{j=1}^r\frac{j}{s}
D_{s-j}\tond{\Psi_j-Z_j}\ ;
\end{displaymath}
indeed the Hamiltonian has initially only one nonlinear term, namely a
quartic polynomial $H_1$, and $\chi_s=0$ for all $s\geq r+1$. 

Let us consider a generic cyclically symmetric polynomial
$F\in\Dscr(C_f,\sigma_h)$ of degree $2h+2$ with $h=1,\ldots,r$. We are
interested in estimating $D_lF$; more precisely, we look for a
sequence $d_l$ such that $D_lF\in\Dscr(d_l C_f,\sigma_*)$. We notice
that, since $\sigma_*<\sigma_r$, it is possible to deal with
$\lie{\chi_r}{(\Psi_r-Z_r)}$, where both functions in the Poisson
brackets belong to $\Dscr(\cdot,\sigma_r)$. For any $1\leq j\leq r$,
one has $\lie{\chi_j}F \in \Dscr(c_j,\sigma_*)$ for a some suitable
$c_j$. We estimate $c_j$ using \eqref{dcdm.11b}
\begin{displaymath}
c_j \leq \frac{4(j+1)(h+1)C_{\chi_j}C_f}{(1-e^{-\max(\sigma_j,\sigma_h)})
(1-e^{-\max(\sigma_j,\sigma_h)+\sigma_*})}\ .
\end{displaymath}
From the first of \eqref{eq.denpois} we have
\begin{displaymath}
\frac{1}{1-e^{-\max(\sigma_j,\sigma_h)}} \leq \frac4{1-e^{-\sigma_0}}\ ,
\end{displaymath}
and
\begin{align*}
  1-e^{-\max(\sigma_j,\sigma_h)+\sigma_*} &\geq
  \tond{\frac{\max(\sigma_j,\sigma_h)-\sigma_*}{\sigma_0-\sigma_*}}
  \tond{1-e^{-(\sigma_0-\sigma_*)}} 
\geq \\
&\geq
  \tond{\frac{\sigma_r-\sigma_*}{\sigma_0-\sigma_*}} \tond{1-e^{-(\sigma_0-\sigma_*)}}
=
  \tond{\frac{\sigma_1-\sigma_*}{\sigma_0-\sigma_*}}
  \tond{\frac{1-e^{-(\sigma_0-\sigma_*)}}{r}}
\end{align*}
hence by inserting $\sigma_1=\sigma_0/2$ (see \eqref{e.sigma.1}) and
making use of $\sigma_*=\sigma_0/4$ (see additional hypothesis in
claims (iv) and (v) of Theorem~\ref{prop.gen}) we obtain
\begin{displaymath}
\frac1{1-e^{-\max(\sigma_j,\sigma_h)+\sigma_*}}
\leq \frac{3r}{1-e^{-(\sigma_0-\sigma_*)}}\ .
\end{displaymath}
By combining the two above estimates we get
\begin{displaymath}
c_j \leq \frac{48(j+1)(h+1)(r+1)}{(1-e^{-\sigma_0})
(1-e^{-(\sigma_0-\sigma_*)})}C_{\chi_j}C_f \ .
\end{displaymath}
If we use\footnote{Since we are making use of $\sigma_*=\sigma_0/4$,
it is $E_{0*}=1/3$.}
\begin{equation*}
  C_{\chi_j} \leq C_r^{j-1} \frac{C_{h_1}}{\gamma j}
  \qquad
  C_r =
  \frac{192r^2 C_{h_1}}{\gamma(1-e^{-\sigma_0}) (1-e^{-(\sigma_0-\sigma_*)})}\ ,
\end{equation*}
then, using $h+1\leq r+1$, we have
\begin{equation}
\label{e.cj.1}
c_j \leq C_r^{j-1}\frac{192 r^2 C_{h_1}}{\gamma(1-e^{-\sigma_0})
(1-e^{-(\sigma_0-\sigma_*)})}C_f = C_r^j C_f\ .
\end{equation}
From \eqref{e.D.s}, the sequence $d_l$ has to satisfy
\begin{displaymath}
d_l C_f \leq \sum_{j=1}^l \frac{j}{l} d_{l-j} c_j
\leq \sum_{j=1}^l \frac{j}{l} d_{l-j}C_r^j C_f\ .
\end{displaymath}
We look for $d_l$ of he form $d_m := \theta_m C_r^m$.  By inserting in
the above recursive inequality and removing the common factor $C_r^l$
we obtain for $\theta_l$ the relation\footnote{$\theta_l$ is dominated
  by $x_l := \sum_{j=0}^{l-1} x_j$, which gives $x_l=2^l$ once it is set
  $x_0=1$.}
\begin{displaymath}
\theta_l \leq \sum_{j=1}^l\frac{j}{l}\theta_{l-j}
< \sum_{j=0}^{l-1}\theta_j \leq 2^l\ ,
\end{displaymath}
thus we have
\begin{equation}
\label{e.C.new.r}
D_l F \in \Dscr(\tilde C_r^l C_f,\sigma_*)\ ,\qquad\qquad
\tilde C_r := 2 C_r\ .
\end{equation}
We apply the above result to the elements giving $\Ham{r}{s}$; for the
sake of brevity we define $F_j:=\Psi_j-Z_j\in\Dscr\tond{\frac{2C_{h_1}}j
  C_r^{j-1},\sigma_*}$, so we have, for all $s\geq r+1$,
\begin{equation*}
\begin{aligned}
  D_{s-1}H_1 &\in \Dscr\tond{C_{h_1} \tilde C_r^{s-1},\sigma_*} 
\\
  D_{s-j}F_j &\in \Dscr\tond{\frac{4C_{h_1}}{j2^j} \tilde C_r^{s-1},\sigma_*}
\end{aligned}
\ \Longrightarrow\ 
C_{\Ham{r}{s}} \leq \tilde C_r^{s-1}
C_{h_1} \quadr{1+\sum_{j=1}^r\frac4{s2^j}} < 2\tilde C_r^{s-1}
C_{h_1}\ .
\end{equation*}

%
%

\paragraph{The canonical transformation.}

In order to control the domain of validity of the normal form, we need
to estimate the deformation due to the canonical transformation
$\T_\Chi$. Since it is obtained by composition of $r$ consecutive
normal form steps $\T_r := \T_{\Chi_r}$, we set
\begin{displaymath}
d=r\delta_1\qquad\qquad \delta_s=\delta_1=\frac{d}{r}\qquad \forall
s=1,\ldots,r\ ,
\end{displaymath}
and we require for any $s=1,\ldots,r$ the generic transformation
$\T_s$ to map
\begin{displaymath}
\T_s:B_{R_s}\rightarrow B_{R_{s-1}}\ ,
\qquad\text{where}\quad
R_s := R_{s-1}-\delta_s\ , \quad R_0\equiv R \ .
\end{displaymath}
Hence, by composition, the whole transformation $\T$ maps $B_{R-d}$ to
$B_R$. We have

\begin{lemma}
\label{l.6}
Let $\delta_s<R_{s-1}$ and $\ncamp{X_{\Xscr_s}}_{R_{s-1}}<\delta_s/e$,
then for $|t|\leq 1$ and $\norm{z}\leq R_{s}$ one has
\begin{equation}
  \label{e.defr0}
  \norm{\T_{s}^t(z) - (z)} \leq
    \frac1{1-c_s}\ncamp{X_{\Xscr_s}}_1\norm{z}^{2s+1} \ ,
\qquad\text{with}\qquad
  c_s:=\frac{e}\delta_s\ncamp{X_{\Xscr_s}}_{R_{s-1}} \ .
\end{equation}
\end{lemma}

\begin{proof}
By series expansion we have
\begin{displaymath}
  \T_s^t(z) - z = \sum_{p\geq 1}\frac1{p!}\lie{\Xscr_s}^{p-1}X_{\Xscr_s}(z) =
  \sum_{p\geq 0}\frac1{(p+1)!}\lie{\Xscr_s}^{p}X_{\Xscr_s}(z) \ .
\end{displaymath}
By applying Corollary~\ref{c.0} (with both $\Chi$ and $F$ of the
Corollary equal to $\Chi_s$)
\begin{equation}
\label{e.l6.1}
\norm{\lie{\Chi_s}^{p}X_{\Xscr_s}(z)}\leq
\prod_{j=0}^{p-1}[(j+1)s-j]\left(\ncamp{X_{\Xscr_s}}_1\right)^{p+1}
\norm{z}^s\norm{z}^{p(s-1)} \ .
\end{equation}
Let us take initially the case $p=2$ with $\norm{z}\leq R-2d$ and
$d=\delta_s/2$; we apply twice \eqref{e.cauchy} and \eqref{e.c.3.1}
\begin{displaymath}
(2s-1)s(R-2d)^{2s-2}\leq \frac1d s (R-d)^{2s-1} \leq \frac1d R^s s
  (R-d)^{s-1}\leq \frac1{d^2}R^{2s}.
\end{displaymath}
We define $d=\delta_s/p$ and by iteration of the same argument we get
\begin{equation}
\label{e.l6.2}
\prod_{j=0}^{p-1}[(j+1)s-j](R-pd)^{ps-p}\leq \frac1{d^p}R^{ps} =
\frac{p^p}{\delta_s^p}R^{ps}\leq p!\tond{\frac{e}{\delta_s}}^p (R^s)^p.
\end{equation}
We can estimate \eqref{e.l6.1} as follows
\begin{displaymath}
\norm{\lie{\Chi_s}^{p}X_{\Xscr_s}(z)}\leq
p!\ncamp{X_{\Xscr_s}}_1\norm{z}^s\tond{\frac{e\ncamp{X_{\Xscr_s}}_R}{\delta_s}}^p.
\end{displaymath}
Thus for $\norm{z}\leq R-\delta_s$ we have
\begin{displaymath}
\norm{\sum_{p\geq 0}\frac1{(p+1)!}\lie{\Xscr_s}^{p}X_{\Xscr_s}(z)} \leq
\norm{X_{\Xscr_s}}_1\norm{z}^s\sum_{p\geq 0}
 \tond{\frac{e}\delta_s \ncamp{X_{\Xscr_s}}_R}^{p}\ ,
\end{displaymath}
which gives the claim.
\end{proof}

\begin{corollary}
\label{c.4}
Let $\delta_s<R_{s-1}$ and also assume
\begin{equation}
\label{e.sm.Xchis}
\ncamp{X_{\Xscr_s}}_{R_{s-1}}<\frac{\delta_s}{1+e}\ ,
\end{equation}
then, for $z\in B_{R_s}$, one has
\begin{equation}
\label{e.def.chis}
\norm{\T_{s}(z)-z}\leq (1+e)\ncamp{X_{\Xscr_s}}_{R_{s-1}} \leq \delta_s
\qquad\Rightarrow\qquad
\T_{s}:B_{R_s}\to B_{R_{s-1}} \ .
\end{equation}
\end{corollary}

\begin{proof}
Given the control of $z$ (i.e. $\norm{z}\leq R_s < R_{s-1}$),
and the degree of $X_{\Xscr_s}$, one has
\begin{equation*}
\ncamp{X_{\Xscr_s}}_1\norm{z}^{2s+1} <  \ncamp{X_{\Xscr_s}}_{R_{s-1}} \ ;
\end{equation*}
by the definition of $c_s$ in~\eqref{e.defr0} and by
hypothesis~\eqref{e.sm.Xchis}, one has $c_s\leq \frac{e}{e+1}$,
i.e. $\frac1{1-c_s}\leq 1+e$; inserting these inequalities into the
estimate of~\eqref{e.defr0} one gets the thesis.
\end{proof}

%
%

\paragraph{Conclusion of the proof.} 

We now make use of Corollary \ref{c.4} to conclude the proof of
(iv). From \eqref{e.Hamfield} we have
\begin{displaymath}
\ncamp{X_{\Xscr_s}}_{R_{s-1}}\leq \ncamp{X_{\Xscr_s}}_R \leq
\frac{4(s+1)R^{2s+1}C_{\chi_s}}{(1-e^{-\sigma_s})^2}\ ;
\end{displaymath}
we apply \eqref{e.exp.ineq} 
\begin{align*}
1-e^{-\sigma_s}
&> \tond{\frac{\sigma_s}{\sigma_0-\sigma_s}}(1-e^{-(\sigma_0-\sigma_*)})
> \tond{\frac{\sigma_*}{\sigma_0-\sigma_*}}(1-e^{-(\sigma_0-\sigma_*)})\\
1-e^{-\sigma_s} &> \tond{\frac{\sigma_s}{\sigma_0}}(1-e^{-\sigma_0})
> \tond{\frac{\sigma_*}{\sigma_0}}(1-e^{-\sigma_0})
\end{align*}
and inserting $\sigma_*=\sigma_0/4$ we get
\begin{displaymath}
\frac1{(1-e^{-\sigma_s})^2}
< \frac{\sigma_0(\sigma_0-\sigma_*)}{\sigma_*^2(1-e^{-\sigma_0})(1-e^{-(\sigma_0-\sigma_*)})} =
\frac{12}{(1-e^{-\sigma_0})(1-e^{-(\sigma_0-\sigma_*)})}\ ,
\end{displaymath}
thus the smallness condition can be replaced by
\begin{displaymath}
\frac{48(s+1)R^{2s+1}C_{\chi_s}}{(1-e^{-\sigma_0})(1-e^{-(\sigma_0-\sigma_*)})} <
\frac{\delta_s}{1+e}\ .
\end{displaymath}
We recall that
\begin{displaymath}
C_{\chi_s} = \frac1{s\gamma} C_{h_1}C_r^{s-1}\qquad\qquad C_r
= \frac{192r^2 C_{h_1}}{\gamma(1-e^{-\sigma_0})
(1-e^{-(\sigma_0-\sigma_*)})}\ ,
\end{displaymath}
so that the field $X_{\Chi_s}$ fulfills
\begin{equation}
\label{e.bound.chis}
\ncamp{X_{\Xscr_s}}_{R_{s-1}}\leq
\frac{48(s+1)R^{2s+1}C_{\chi_s}}{(1-e^{-\sigma_0})(1-e^{-(\sigma_0-\sigma_*)})}
\leq \frac{R}{2r^2}a_s \ ,
\qquad\text{with}\quad
a_s := \tond{R^2 C_r}^{s} \ .
\end{equation}
Under the smallness condition \eqref{e.R.sm1}, the sequence $a_s$ is
controlled by a geometrically decreasing one, $a_s <
\tond{\frac{2}{3(1+e)}}^s$, so that we can think at $\T_s$ as a
sequence of increasingly smaller deformation of the identity, the
first $\T_1$ being the biggest.

With the choice $d:=\frac{R}3$, condition \eqref{e.sm.Xchis} is
ensured by imposing
\begin{displaymath}
\frac{R}{2 r^2}a_1 < \frac{\delta_s}{1+e} = \frac{\delta_1}{1+e}  =  
\frac{d}{r(1+e)} =\frac{R}{3r(1+e)} \ ,
\end{displaymath}
which is fulfilled provided $a_1 < \frac{2}{3(1+e)}$, which in turn is
\eqref{e.R.sm1}.

To conclude the proof of Theorem~\ref{prop.gen} we still have to prove
the smallness of the deformation of the domain $B_{\frac23 R}$. This
is obtained summing up all the (geometrically decreasing) deformations
\eqref{e.def.chis} of the iteratively defined domains $B_s$. Indeed,
we exploit \eqref{e.bound.chis} to collect all the $r$ deformations
\begin{displaymath}
\norm{T_\Chi(z)-z}\leq \frac{(1+e)R}{2r^2}\quadr{\sum_{s=1}^r a_s}
\leq \frac{(1+e)R^3 C_r}{r^2}< R^3 4^4 C_*\ .
\end{displaymath}

%
%

\subsection{Proof of Corollary \ref{c.Hom.K.var}}

Since by hypothesis $z(0)\in B_{\frac{R}9}$, the transformed initial
datum lies in ${\tilde z}(0)\in B_{\frac{4}9R}$. Let $\tau$ the escape
time from the ball of radius $\frac23 R$, then for all $|t|<\tau$ it
holds
\begin{align*}
  |H_\Omega({\tilde z}(t))-H_\Omega({\tilde z}(0))| 
\leq
  \int_0^t|\Poi{H_\Omega}{P^{(r+1)}} {\tilde z}(s)|ds 
\leq&
  \norm{X_{H_\Omega}}_{\frac23 R} \norm{X_{P^{(r+1)}}}_{\frac23 R} |t|
=\\=&
  \frac23 \Omega R \norm{X_{P^{(r+1)}}}_{\frac23R} |t|\ .
\end{align*}
In order to deal with the remainder, we apply \eqref{e.rem.r} 
\begin{displaymath}
P^{(r+1)} = \sum_{s\geq r+1} \Ham{r}{s}\ ,\qquad\qquad
\norm{X_{\Ham{r}{s}}}_{\frac23 R} = \ncamp{X_{\Ham{r}{s}}}_1
\tond{\frac23 R}^{2s+1}\ ;
\end{displaymath}
from Theorem~\ref{prop.gen}, (v), and Lemma \ref{lem.Hamfield} we
know that
\begin{displaymath}
\ncamp{X_{\Ham{r}{s}}}_1 \leq \frac{8(s+1)}{(1-e^{-\sigma_*})^2}
C_{h^{(r)}_s}\ ,
\end{displaymath}
hence the Hamiltonian field of the whole remainder $P^{(r+1)}$ fulfills
\begin{displaymath}
\norm{X_{P^{(r+1)}}}_{\frac23 R}\leq \frac{16 C_{h_1}}{(1-e^{-\sigma_*})^2}
R^{2r+3}\tilde C_r^r \quadr{\sum_{h\geq 0}(h+3)\tond{\frac23 R^2 \tilde
    C_r}^h}\ .
\end{displaymath}
From the smallness assumption \eqref{e.R.sm1} we have $R^2 \tilde
C_r<1/(1+e)$, then
\begin{displaymath}
|H_\Omega({\tilde z}(t))-H_\Omega({\tilde z}(0))| \leq \frac{2^5 \Omega
  C_{h_1}}{(1-e^{-\sigma_*})^2} \tond{\frac23 R}^{2r+4} \tilde C_r^r |t|\ ,
\end{displaymath} 
which gives for a suitable $C$
\begin{displaymath}
|H_\Omega({\tilde z}(t))-H_\Omega({\tilde z}(0))| < \Omega R^4\ ,\qquad
|t|\leq \frac{C (1-e^{-\sigma_*})^2}{ C_{h_1}} \tond{\frac23
  R^2 \tilde C_r}^{-r}\ .
\end{displaymath}
The variation in the original coordinates follows from two facts. The
first and main one is that by controlling $H_\Omega$ we are
controlling the escape time from the ball where we started. The second
one is the deformation of the canonical transformation $T_\Chi$, which
according to \eqref{e.def.Tchi} gives
\begin{displaymath}
|H_\Omega({\tilde z})-H_\Omega(z)| \leq \Omega R^4\ .
\end{displaymath}

We proceed in the same way to control the variation of $\Z$. From its
definition in \eqref{e.Ham.r}
\begin{displaymath}
\ncamp{X_{\Z}}_{\frac23 R}\leq \sum_{s=0}^r \ncamp{X_{Z_s}}_{\frac23
  R}\ ;
\end{displaymath}
we make use of Lemma \ref{lem.Hamfield} to get
\begin{displaymath}
\ncamp{X_{Z_0}}_1 \leq \frac{8 C_{\zeta_0}
  \mu}{(1-e^{-\sigma_*})^2}\ ,\qquad\qquad \ncamp{X_{Z_s}}_1 \leq
\frac{8 C_{h_1} C_r^{s-1}}{(1-e^{-\sigma_*})^2}\qquad s=1,\ldots,r\ .
\end{displaymath}
where the factor $\mu$ in the first estimate follows from the fact
that $\zeta_0^{(0)}=0$.
\qed

%
%
%
%
%

\section{Appendix}
\label{s:6}

\subsection{Poisson brackets of cyclically symmetric polynomials}

The following Lemma produces a general estimate of the Poisson bracket
specially adapted to the case of cyclically symmetric polynomials. It
is crucial for the control of the dependence on $N$ of the norms of
``extensive'' functions generated by our perturbation scheme.  Its
proof, and those of the subsequent statements, are collected in the
Appendix of \cite{GioPP13}.

\begin{lemma}[see \cite{GioPP13}]
\label{l.2.2}
Let $f(x,y)$ and $g(x,y)$ be homogeneous polynomials respectively of
degree $r$ and $s$.  Then $\{f,g\}$ is a homogeneous polynomial of
degree $r+s-2$, and one has
\begin{displaymath}
\|\{f,g\}\|_1 \le rs \|f\|_1\, \|g\|_1\ .
\end{displaymath}
Moreover, there exists a seed of $\{f^{\oplus},g^{\oplus}\}$ such that
one has
\begin{equation}
\label{e.poisson.1}
\bigl\|\{f^{\oplus},g^{\oplus}\}\bigr\|_1^{\oplus} \le 
 rs \bigl\|f^{\oplus}\bigr\|_1^{\oplus}\, \bigl\|g^{\oplus}\bigr\|_1^{\oplus}.
\end{equation}
\end{lemma}

The next statements provide the basic estimates for controlling 
the exponential decay.

\begin{lemma}[see \cite{GioPP13}]
\label{lem.poisson}
Let $F,\,G$ be cyclically symmetric homogeneous polynomials of degree
$r',r''$ respectively.  Let the seeds $f,\,g$ be of class
$\Dscr(C_f,\sigma')$ and $\Dscr(C_g,\sigma'')$, respectively, and let
$\sigma\lt\min(\sigma',\sigma'')$.  Then there exists $C_h\ge 0$ such
that the seed $h$ of $H=\Poi{F}{G}$ is of class $\Dscr(C_h,\sigma)$.
An explicit estimate is
\begin{equation}
\label{dcdm.11b}
C_h = \frac{r' r'' C_f
  C_g}{(1-e^{-\max(\sigma',\sigma'')})(1-e^{-\max(\sigma',\sigma'')+\sigma})}\ .
\end{equation}
\end{lemma}

\begin{corollary}[see \cite{GioPP13}]
\label{cor.poisson}
If in lemma~\ref{lem.poisson} we have $\sigma'\neq\sigma''$ then we
may set $\sigma=\min(\sigma',\sigma'')$ and
\begin{equation*}
C_h = \frac{r' r'' C_f
  C_g}{(1-e^{-\max(\sigma',\sigma'')})(1-e^{-|\sigma'-\sigma''|})}\ .
\end{equation*}
\end{corollary}

\begin{corollary}[see \cite{GioPP13}]
\label{cor.pois.Z0}
If in lemma~\ref{lem.poisson} we have $\sigma'\gt\sigma''$  and
$f^{(0)}=0$, i.e., $f=\sum_{m\ge 1}f^{(m)}=O(e^{-\sigma'})$ then we
may set $\sigma=\sigma''$ and
\begin{equation}
\label{dcdm.10a}
C_h = \frac{2e^{-(\sigma'-\sigma'')}r' r'' C_f
  C_g}{(1-e^{-\sigma'})(1-e^{-(\sigma'-\sigma'')})}\ .
\end{equation}
\end{corollary}


\subsection{Proof of Proposition \ref{p.field}}
\label{aa:pfield}

\begin{proof}
We need to prove~\eqref{e.est0} both in the cases of euclidean and
supremum norm. We start with the latter, which is easier.
Let us set $\norm{z}=\norm{z}_\infty$.  Due to \eqref{e.seme.campo},
we can write
\begin{displaymath}
|X_j(z)| < |\tau^{j-1}X_1(z)|+|\tau^{j-1}X_{N+1}(z)|\ .
\end{displaymath}
By expanding $X_1$ and $X_{N+1}$ as polynomials, we have
\begin{displaymath}
|X_1(z)|\leq \sum_{k}|X_{1,{k}}||z^{k}| \leq \norm{z}^r\norm{X_1}_1\ ,
\qquad
|X_{N+1}(z)|\leq \norm{z}^r\norm{X_{N+1}}_1\ .
\end{displaymath}
Since, for all $j$, one has $|\tau^{j-1}(z_1^{k_1}\cdots z_{2N}^{k_{2N}})|\leq \norm{z}^r$,
it follows
\begin{displaymath}
|X_j(z)| < \ncamp{X}_1\norm{z}^r\ ,
\end{displaymath}
which gives \eqref{e.est0}.

Let us now set $\norm{z}=\norm{z}_2$.  We are interested in
\begin{displaymath}
\norm{(X(z),Y(z))}^2 = \sum_{l=1}^N\tond{|X_l(z)|^2+|Y_l(z)|^2} =
\sum_{l=1}^N\tond{|\tau^{l-1}X_1(z)|^2+|\tau^{l-1}Y_1(z)|^2}\ .
\end{displaymath}
Since $X_1(z)$ is a polynomial of degree $r$ one has 
$X_1(z) = \sum_{|j|=r}X_{1,j} z^j$,
where $j=(j_1,\ldots,j_n)$ is a $n=2N$ multi-index.
We write
\begin{eqnarray}
\sum_{l=1}^N\big|X_l(z)\big|^2 
= 
\sum_{l=1}^N\Bigg|\sum_{|j|=r}X_{1,j}(\tau^{l-1}z)^j\Bigg|^2 
\label{e.part1}&=&
\sum_{l=1}^N\sum_{|j|=r}X_{1,j}^2(\tau^{l-1}z)^{2j} +
\\
\label{e.part2}&+&
\sum_{l=1}^N\sum_{j\not=h\atop{|j|=|h|=r}}X_{1,j}X_{1,h}(\tau^{l-1}z)^{j+h} \ .\qquad\null
\end{eqnarray}
Let us split the estimates of \eqref{e.part1} and \eqref{e.part2}.  In
the former term we invert the order of the sums over $l$ and $j$;
recalling that $\tau$ acts separately on $x$ and $y$, we have
\begin{align*}
\sum_{l=1}^N&(\tau^{l-1}z)^{2j} = z_1^{2j_1}\cdots z_n^{2j_n} +
(z_2^{2j_1}\cdots z_1^{2j_N})(z_{N+2}^{2j_{N+1}}\cdots
z_{N+1}^{2j_{2N}}) + \ldots 
+ \\
& + (z_N^{2j_1}\cdots
z_{N-1}^{2j_N})(z_{2N}^{2j_{N+1}}\cdots z_{2N-1}^{2j_{2N}}) = 
\sum_{l=1}^N\tond{\tau^{l-1}z^2}^{j} 
< \tond{\sum_{l=1}^n z_l^2}^{|j|} = \norm{z}^{2r}
\ ,
\end{align*}
with a rough (but uniform in $j$) estimate; thus we have
\begin{displaymath}
\sum_{|j|=r}|X_{1,j}|^2\tond{\sum_{l=1}^N|\tau^{l-1}z|^{2j}}\leq
\norm{z}^{2r}\sum_{|j|=r}|X_{1,j}|^2.
\end{displaymath}

Let us now come the second term, \eqref{e.part2}, and rewrite it as
\begin{displaymath}
\sum_{|j|=|h|=r,j\not=h}X_{1,j}X_{1,h}\sum_{l=1}^N(\tau^{l-1}z^{k}),\qquad
j+h = k,\qquad |k|=2r.
\end{displaymath}
The idea is to provide again an estimate like
$\sum_{l=1}^N|\tau^{l-1}z^{k}|\leq \norm{z}^{2r}$; the tricky point is
the possible presence of odd exponent in the multiindex
$k=(k_1,\ldots,k_n)$.  We then first decompose $k$ in its ``even'' and
``odd'' parts
\begin{displaymath}
k = 2k^\natural + k^\sharp,\qquad |k^\natural|\leq r,\qquad |k^\sharp|
= 2(r-|k^\natural|) =: 2s, \qquad k^\sharp_j\in \{0,1\} \ ,
\end{displaymath}
and consequently decompose the monomial $z^k$ as $z^k =
z^{2k^\natural} z^{k^\sharp}$; then we rewrite explicitly
\begin{displaymath}
z^{k^\sharp} = z_{i_1}\cdots z_{i_{2s}},\qquad\qquad
i_1,\dots,i_{2s}\in {\cal J},
\end{displaymath}
where ${\cal J}$ represents the subset of those indexes
$i_l\in\{1,\ldots,n\}$ such that $k^\sharp_{i_l}=1$.  So one has
\begin{eqnarray*}
|z_{i_1}\cdots z_{i_{2s}}| &\leq&
\frac12(z_{i_1}^2+z_{i_{2}}^2)\cdots\frac12(z_{i_{2s-1}}^2+z_{i_{2s}}^2)
= \frac1{2^s}\prod_{m=1}^s\tond{z_{j_{2m-1}}^2 + z_{j_{2m}}^2} = \\
&=& \frac1{2^s}\sum_{\substack{l_m\in\{2m-1,2m\}\\m=1,\ldots,s}}{z^2_{l_1}\cdots
  z^2_{l_s}}\leq \norm{z}^{2 s}\ ,
\end{eqnarray*}
where last sum has exactly $2^s$ elements. The above upper bound holds
also for any translated monomial $|\tau^{l-1}z^{k^\sharp}|$.

From the above considerations we get
\begin{displaymath}
\sum_{l=1}^N|\tau^{l-1}z^{k}| =
\sum_{l=1}^N|\tau^{l-1}z^{2k_\natural}||\tau^{l-1}z^{k_\sharp}| 
\leq \norm{z}^{2 s}\sum_{l=1}^N|\tau^{l-1} z^{2k_\natural}| \leq
\norm{z}^{2r}\ ,
\end{displaymath}
thus collecting the diagonal and off-diagonal elements of
$\left(\ncamp{X_1}_1\right)^2$ we obtain
\begin{displaymath}
\norm{X(z)}^2\leq \norm{z}^{2r}\tond{\sum_{|j|=r}X_{1,j}^2 +
  \sum_{|j|=|h|=r,j\not=h}|X_{1,j}||X_{1,h}|} =
\norm{z}^{2r}\tond{\sum_{|j|=r}|X_{1,j}|}^2.
\end{displaymath}
Using a similar estimate for the component $Y$, we finally get
\begin{equation}
\label{e.est1}
\norm{F(z)}\leq \ncamp{F}_1\norm{z}^r\ ,
\end{equation}
and the thesis follows.
\end{proof}


\subsection{Proof of Lemma \ref{lem.Hamfield}}
\label{aa:Hamfield}

\begin{proof}
We will use \eqref{e.xxx}. By the hypothesis on the seed $f$ we have
$f=\sum_{m\geq 0}f^{(m)}$ with $\norm{f^{(m)}}_1 \leq C_{f}e^{-\sigma
  m}$.  Using both the fact that $f^{(m)}$ depends only on the subsets
$x_0,\ldots,x_m$ and $y_0,\ldots,y_m$, and it decays with $m$, one has
\begin{equation*}
\begin{aligned}
  \sum_{l=0}^{N-1} \norm{\derp{f}{x_l}}_R &\leq
  \sum_{m\geq 0}\sum_{l=0}^{N-1} \norm{\derp{f^{(m)}}{x_l}}_R =
  \sum_{m\geq 0}\sum_{l=0}^{m} \norm{\derp{f^{(m)}}{x_l}}_R \leq
\\
 & \leq  \sum_{m\geq 0}\sum_{l=0}^{m} r R^{r-1}C_f e^{-\sigma m} \leq
  r R^{r-1}\frac{2C_f}{(1-e^{-\sigma})^2}\ ,
\end{aligned}
\end{equation*}
where we also used
\begin{equation*}
\begin{aligned}
  \sum_{m\geq 0}(m+1)e^{-\sigma m} &=
  1 + \sum_{m\geq 1}(m+1)e^{-\sigma m} \leq
  1 + \int_1^{+\infty}(x+1)e^{-\sigma (x-1)}dx =
  \\
  &= \frac{\sigma^2+2\sigma + 1}{\sigma^2}
  \leq \frac2{(1-e^{-\sigma})^2} \ .
\end{aligned}
\end{equation*}
The same calculation holds for derivatives with respect to the $y$ variables.
\end{proof}


\subsection{Lie derivative of a vector field}

\begin{lemma}
\label{l.0.2}
Let $\Chi=\chi^\oplus$ a cyclically symmetric polynomial of degree
$r+1$ and $X_F$ an Hamiltonian vector field of a cyclically symmetric
Hamiltonian $F=f^\oplus$, where $f$ is a polynomial of degree
$s+1$. Then it holds true
\begin{equation}
\label{e.l.0.2}
\norm{L_\Chi X_F(z)}\leq
s\ncamp{X_F}_1\ncamp{X_\Chi}_1\norm{z}^{s+r-1}.
\end{equation}
\end{lemma}

\begin{proof}
As already remarked in \eqref{e.multilin.X}, we can interpret $X_F$ as
an $r$-linear operator, hence there exists $\tilde X_F$ such that
$X_F(z) = \tilde X_F(z,\ldots,z)$; we can thus write the Lie
derivative of $X_F$ as
\begin{displaymath}
L_\Chi X_F(z) = dX_F(z)[X_\Chi(z)] = s\tilde
X_F(X_\Chi(z),z,\ldots,z)\ ,
\end{displaymath}
which, using \eqref{e.est1}, gives the thesis.
\end{proof}

\begin{corollary}
\label{c.0}
Under the same hypothesis of Lemma \ref{l.0.2} it holds true
\begin{equation}
\label{e.c.0}
\norm{L_\Chi^p X_F(z)}\leq \prod_{j=0}^{p-1}[s+j(r-1)] \ncamp{X_F}_1
\tond{\ncamp{X_\Chi}_1}^p \norm{z}^{s+p(r-1)}\ .
\end{equation}
\end{corollary}

\begin{corollary}
\label{c.3}
Under the same hypothesis of Lemma \ref{l.0.2}, if $\norm{z}\leq
R-\delta$, it holds true
\begin{equation}
\label{e.c.3.1}
\norm{L_\Chi X_F(z)}\leq
\frac{\ncamp{X_F}_R}{\delta}\ncamp{X_\Chi}_1\norm{z}^r.
\end{equation}
\end{corollary}

\begin{proof}
  Splitting the term $\norm{z}^{r+s-1}$ of \eqref{e.l.0.2} in
  $\norm{z}^{s-1}\norm{z}^r$, and using\footnote{defining
    $x=\delta/R<1$, we may rewrite \eqref{e.cauchy} as
    $g_r(x):=rx(1-x)^{r-1}<1$, which is true since for $x\in[0,1]$ we
    have $g_r(x)\leq g_r(1/r)=\tond{1-\frac1r}^{r-1}<1$.  }
\begin{equation}
\label{e.cauchy}
s(R-\delta)^{s-1} < \frac1\delta R^s
\end{equation}
in \eqref{e.l.0.2} we get the thesis.
\end{proof}


\subsection{Proof of Lemma \ref{l.H1.sym.exp}}

Let us write\footnote{{If $2/N$ one has to set $(\delta q)_{N/2} =
    q_{N/2}$.}}
\begin{displaymath}
x_0 = a_0q_0 + \sum_{m=1}^{[N/2]}a_m (\delta q)_m \ ,
\qquad
(\delta q)_m:=(q_m+ q_{N-m}) \ ,
\qquad
a_m := A_{1,m}^{-1/4} \ ,
\end{displaymath}
so that $S\tond{(\delta q)_m}\subset [0,\ldots,m]\cup[N-m,\ldots,N]$.
Then set ${\A}:= a_0q_0$ and ${\B}:=x_0 -{\A}$. Thus we expand the
seed $x_0^4$ as $ ({\A}+{\B})^4 = {\A}^4 + 4{\A}^3{\B} +6{\A}^2{\B}^2
+ 4{\A}{\B}^3 + {\B}^4$, and we deal separately with the five terms of
the expansion of the seed. We list below which monomials extracted
from $x_0^4$, are going to compose the element $h_1^{(m)}$.  We assume
$N$ odd, the even case follows almost identically.  When $m=0$ we
plainly have $h_1^{(0)} = {\A}^4 = a_0^4q_0^4$. For $m=1,\ldots,[N/2]$
we have:

\begin{description}
\item[${\A}^3{\B}$ {\bf term:}] since
\begin{displaymath}
{\A}^3{\B} = \sum_{m=1}^{[N/2]} a_0^3a_m q_0^3(\delta q)_m\ ,
\end{displaymath}
with
\begin{displaymath}
S\tond{q_0^3(\delta q)_m}\subset [0,\ldots,m]\cup[N-m,\ldots,N]\ ,
\end{displaymath}
we take only $a_0^3a_mq_0^3(\delta q)_m$. It gives $|a_0^3a_m| =
\mathcal{O}\tond{e^{-\sigma_0 m}}$;

\item[${\A}^2{\B}^2$ {\bf term:}] we take
\begin{displaymath}
a_0^2q_0^2a_m(\delta q)_m\quadr{a_m(\delta q)_m + 2\sum_{1\leq
    i<m}a_i(\delta q)_i}\ ;
\end{displaymath}
by using the additional decay of $|a_i|$, it gives
$a_0^2a_m^2 + 2\sum_{1\leq i<m}|a_0^2a_ia_m| =
\mathcal{O}\tond{e^{-\sigma_0 m}}$;

\item[${\A}{\B}^3$ {\bf term:}] we take
\begin{align*}
a_0a_m^3 q_0(\delta q)_m^3 &+ 3 a_0 q_0\quadr{\sum_{i<m}a_ia_m^2 (\delta
  q)_i (\delta q)_m^2 + \sum_{i<m}a_i^2a_m (\delta q)_i^2 (\delta
  q)_m} +\\ &+6 a_0 q_0\quadr{\sum_{i<j<m} a_ia_ja_m (\delta q)_i (\delta q)_j
  (\delta q)_m}\ ;
\end{align*}
by using the additional decay of $|a_i|$ and $|a_j|$ it gives
$|a_0a_m^3| + 3\sum_{i<l}|a_0a_ia_m^2|+3\sum_{i<l}|a_0a_i^2a_m|
+6\sum_{i<j<l}|a_0a_ia_ja_m| = \mathcal{O}\tond{e^{-\sigma_0 m}}$;

\item[${\A}^4$ {\bf term:}] we take
\begin{align*}
a_m^4 (\delta q)_m^4 &+ 6\sum_{i<l}a_i^2a_m^2 (\delta q)_i^2 (\delta
q)_m^2 + \\ &+ 12\sum_{i,\,j<l}a_ia_j^2a_m (\delta q)_i(\delta
q)_j^2(\delta q)_m + 12\sum_{i,\,j<l}a_ia_ja_m^2 (\delta q)_i(\delta
q)_j(\delta q)_m^2 +\\ &+24\sum_{i<j<h<l}a_0a_ia_ha_m (\delta
q)_0(\delta q)_i(\delta q)_h(\delta q)_m\ ,
\end{align*}
giving a contribute $\mathcal{O}\tond{e^{-\sigma_0 m}}$.
\end{description}

This concludes the proof.


\paragraph{Acknowledgement.}
We warmly thank Dario Bambusi for many useful discussions and
comments. We are indebted with both the referees for their
constructive criticism, which helped us to clarify many aspect of the
manuscript during the revision. T.P. would like also to thank
V.Koukouloyannis, P.G.Kevrekidis and D.Pelinovsky for many interesting
and helpful discussions about the present paper, during the AIMS14
Conference in Madrid. This research is partially supported by
MIUR-PRIN program under project 2010 JJ4KPA (``Te\-orie geome\-triche
e anali\-tiche dei sistemi Hamilto\-niani in dimensioni finite e
infi\-nite'').



\def\cprime{$'$} \def\i{\ii}\def\cprime{$'$} \def\cprime{$'$}

\end{document}